\documentclass[a4paper,10pt]{article}
\usepackage{amsfonts}
\usepackage{graphicx,amssymb,amsmath,amscd}
\usepackage[english]{babel}
\usepackage[all]{xy}
\usepackage{latexsym}
\usepackage{amsmath,amsthm}
\usepackage{stmaryrd}
\usepackage{enumitem}
\usepackage{etoolbox}

\usepackage{hyperref}%
\hypersetup{
urlcolor = red,colorlinks = true,
linkcolor = blue,
citecolor = blue,
linktocpage = true,
pdftitle = {Flows and Homotopies of Banach Lie algebroids},
pdfauthor = {Pelletier, Schmeding},
bookmarksopen = true,
bookmarksopenlevel = 1,
unicode = true,
hypertexnames =true
}%
\usepackage{cleveref}
\usepackage{xcolor}
\usepackage{tikz-cd}

\newcommand{\field}[1]{\mathbb{#1}}
\newcommand{\R}{\field{R}}

\newcommand{\N}{\field{N}}
\newcommand{\M}{\field{M}}

\newcommand{\A}{\field{A}}
\newcommand{\E}{\field{E}}

\DeclareMathOperator{\Fl}{Fl}
\newtheorem{theo}{Theorem}[section]

\newtheorem{lem}[theo]{Lemma}
\newtheorem{cor}[theo]{Corollary}
\newtheorem{prop}[theo]{Proposition}
\newtheorem{nota}[theo]{Notations}

\newtheorem{hyp}[theo]{Assumption}

\theoremstyle{definition}
 \newtheorem{setup}[theo]{}
 \newtheorem{defi}[theo]{Definition}
\newtheorem{rem}[theo]{Remark}

\def\dis{\ds}
\def\to{\rightarrow}

\def\g{\gamma}
\def\G{\Gamma}

\def\t{\tau}

\def\f{\flat}

\def\r{\rho}

\def\s{\sigma}

\def\p{\partial}

\def\~{\tilde}
\def\dis{\displaystyle}

\newcommand{\LB}[1][\cdot \hspace{1pt} , \cdot]{[\hspace{1pt} #1 \hspace{1pt} ]}
\newcommand{\DP}[1][\cdot \hspace{1pt} , \cdot]{\left\langle\!\left\langle #1 \right\rangle\!\right\rangle}
\newcommand{\LXu}[1][u]{\mathcal{X}_{\mathfrak{#1}}}
\newcommand{\FlXu}[1][u]{\mathfrak{F}_{\mathfrak{#1}}^{t,s}}
\title{Flows and homotopies of Banach Lie algebroids}
 \author{F. Pelletier\footnote
{in memoriam, formerly Universit\'e de Savoie Mont Blanc}, A. Schmeding\footnote{NTNU Trondheim, Alfred Getz'vei 1, Trondheim, \href{mailto:alexander.schmeding@ntnu.no}{alexander.schmeding@ntnu.no}}}

\date{}

\begin{document}

\maketitle

\begin{abstract}
We establish flow and homotopy tools for Banach Lie algebroids that do not rely on compactly supported extensions. Under suitable hypotheses, we construct parameter-dependent local extensions of sections along algebroid paths, complete lifts and their fibrewise-linear evolutions, and a Bochner-integral variation-of-constants formula. We then prove that an admissible-path homotopy is equivalently a Lie algebroid morphism from the parameter square and a solution of the associated transport equation. This provides the extension, transport and infinitesimal variation theory required for a local integration program of Banach Lie algebroids to be developed in a sequel.
 \end{abstract}



\tableofcontents

\section{Introduction} 
Lie algebroids were introduced in the 1960's, \cite{Pra66,Pra67,Pra68}, as infinitesemial counterparts to a Lie groupoid. Analogous to the classical theory for Lie groups and Lie algebras in finite dimensions it was erroneously claimed that \emph{any Lie algebroid is isomorphic to the Lie algebroid of a Lie groupoid}. In \cite{AlMo85} a counterexample to this assertion was presented. However, there exists a wide example set  of finite dimensional Lie algebroids which can be integrated by a Lie groupoid (see for instance \cite{Daz97,DoLa66, Nis00}).  Infinite-dimensional Lie algebroids were introduced independently in \cite{Ana11} and \cite{Pel12} for manifolds modelled on Banach spaces. This setting  was used in  a series of papers with relation to operator algebraic structures (cf. for instance \cite{Bar10,OSJ14,OJS15,OdSl16,BeOd20}). The reader will also find a basic approach of Banach Lie algebroids and Banach Lie groupoids with some applications in \cite{BGJP19} and \cite{AaGaS20}.

A complete solution to the integrability problem for finite dimensional Lie algebroids was established in \cite{CrFe11}. There a local topological obstruction for a local Lie groupoid integrating the Lie algebroid was constructed. While Lie algebroids do not generally integrate to Lie groupoids, any finite dimensional Lie algebroid is \emph{locally integrable}, in the sense of admitting a local Lie groupoid integrating it. Recently a new proof of this result was produced in \cite{CMS20}. The corresponding arguments rely on extensions of sections, infinitesimal flows, and variation formulas for two-parameter families of admissible paths. The purpose of the present paper is to establish the infinitesimal flow and variation theory required for local integration in the Banach setting. We do not construct an integrating local Banach Lie groupoid here. Instead, we develop the extension, complete-lift, transport, and homotopy results needed to adapt the local integration methods of \cite{CMS20} to Banach Lie algebroids. In a subsequent work, these results will be used to construct a realization form associated with an algebroid spray and to study the resulting local Banach Lie groupoid.

The challenge is, that even standard finite-dimensional arguments do not carry over directly. A general Banach space need not admit smooth bump functions, and a local section of a Banach bundle need not be the restriction of a globally defined section. Moreover, modules of local sections of infinite-rank bundles are not locally finitely generated. A second difficulty concerns extensions of sections along curves. Finite-dimensional proofs frequently use compactly supported extensions of sections or vector fields. Such extensions are unavailable on a general Banach manifold. We instead use linear connections, local additions, and parallel transport to construct extensions for time-dependent algebroid sections. Our results subsume the following.\smallskip

\textbf{Theorem A} (Extension of parameter-dependent sections) \emph{Let $A\to M$ be a Banach vector bundle equipped with suitable linear connections. Every $C^k$-family of sections along a $C^k$-family of curves 
$\gamma \colon [0,1]^2 \rightarrow M$
admits, locally a $C^{k,\infty}$-extension\footnote{the regularity for the extension uses the notion of $C^{r,s}$-mappings from \cite{AaS15,GaS22}. We recall this notion in the beginning of \Cref{timedependenhomotopy}.} by time-dependent local sections. If $M$ is smoothly paracompact, these extensions can be chosen simultaneously over a compact parameter interval.}\smallskip
 
This extension result replaces the use of compactly supported sections in finite-dimensional integration arguments. Its parameter-dependent form is essential for studying variations of admissible paths and for obtaining evolution equations on common domains. The second ingredient is the theory of complete lifts and their fibrewise-linear evolutions. This is motivated  by the central role of linear vector fields in finite-dimensional Lie groupoid theory. There linear fields provide the infinitesimal description of vector bundle automorphisms and derivations. Moreover, under differentiation, multiplicative vector fields on Lie groupoids induce linear vector fields on the associated Lie algebroids, making them intermediary between global and infinitesimal geometric structures, cf. e.g. \cite{Mac05,MaX98}. We consider these objects and their flows in a Banach context. These then satisfy the following.\smallskip 

\textbf{Theorem B} (Complete lifts and transport) \emph{ Let $(A,\rho, \LB_A)$ be a Banach Lie algebroid and let $\mathfrak{u}$ be a local section of $A$. Then $\mathfrak{u}$ admits a unique complete lift $\mathcal{X}_\mathfrak{u}$, whose local flow consists of fibrewise-linear bundle isomorphisms covering the flow of $\rho(\mathfrak{u})$. The induced pullback satisfies 
\[ \left. \frac{d}{dt} \right|_{t=0} (\mathfrak{F}^t_\mathfrak{u})^*v = \LB[\mathfrak{u},v]_A \]
for every local section $v$. For a time-dependent local section $\mathfrak{u}_t$, the associated push-forward transport is the solution operator for }
\[ \frac{\partial}{\partial t}Y_t=\LB[Y_t,\mathfrak{u}_t]_A. \]

This provides the natural framework for variations and homotopies of admissible algebroid paths. In finite dimensions, $A$-paths and $A$-homotopies form the basis of the Weinstein groupoid construction in \cite{CF03}. Closely related infinitesimal variation formulas enter the local Lie groupoid integrating finite-dimensional Lie algebroids approach from \cite{CMS20}. Our main result establishes the corresponding homotopy and variation formula in the Banach setting.
\smallskip

\textbf{Theorem C} (Characterisation of homotopies) \emph{
Let $A\to M$ be a Banach Lie algebroid satisfying the connection hypotheses in Theorem A, and assume that $M$ is smoothly paracompact. Let $a,b\colon [0,1]^2\to A$ be $C^k$-maps with $k\geq2$ satifying the assumptions of \Cref{P_existH}, covering the base map $\gamma$. Then the following assertions are equivalent:
\begin{enumerate} 
\item \((a,b)\) is an $A$-homotopy,
\item the map \[ f_{a,b}:T([0,1]^2)\to A, \qquad f_{a,b}(\lambda\partial_t+\mu\partial_\varepsilon) = \lambda a+\mu b, \] 
is a Lie algebroid morphism over $\gamma$,
\item for one, and hence every, time-dependent local extension $u_\varepsilon$ of $a_\varepsilon$, 
\[ b_\varepsilon(t) = \int_0^t \mathfrak{F}_{\mathfrak{u}_\varepsilon,\gamma_\varepsilon(s)}^{t,s} \left( \partial_\varepsilon \mathfrak{u}_\varepsilon (s,\gamma_\varepsilon(s)) \right)\,ds. \] 
\end{enumerate}}

The significance of the third characterization is that it expresses the transverse component of a homotopy entirely through the evolution of time-dependent complete lifts.
Theorems A and B provide the analytic and infinitesimal machinery for Theorem C. 

Our approach differs from a formal infinite-dimensional translation of the finite-dimensional proofs. The absence of compactly supported extensions is handled by connection-based extensions on neighbourhoods of parameterized graphs. The infinite rank of the bundles requires a formulation in terms of local sections rather than locally finitely generated modules. Finally, the transport equations are treated through operator-norm evolution and fibrewise Bochner integration, so no topology on the full space of local sections is required.
The present paper should therefore be viewed as the infinitesimal and analytic part of a two-step local integration program. Here we establish the extension, flow, transport, and homotopy machinery. In a subsequent work, we use this machinery to construct a realization form from a Banach Lie algebroid spray and to carry out the corresponding local-integration construction in the spirit of \cite{CMS20,CMS22}. Questions concerning global integrability and its possible obstructions are not addressed in the present paper.

\paragraph{Organisation of the article} \Cref{sect:fundamentals} collects the necessary material on linear connections, local additions, Banach Lie algebroids, and their morphisms. \Cref{timedependenhomotopy} constructs time-dependent extensions and complete lifts, develops their transport theory, and proves the equivalent characterizations of homotopy. \Cref{app:smooth_para} contains some topological concepts relevant for infinite-dimensional manifolds. \Cref{app:proof} collects several technical proofs. \Cref{app:dualfield} records the optional dual and partial-Poisson description of the complete lift.

\paragraph{Acknowledgements} This work began in 2025 after Professor Fernand Pelletier approached me with the proposal to collaborate on a set of ideas he had to establish (local) integrability for Banach Lie algebroids. The basic structure and ideas developed in this work are all due to Professor Pelletier. It is unfortunate that he could not see the completion of his ideas. 

\paragraph{Tool and computational resource disclosure:} In preparing the manuscript standard LLM models (ChatGPT 5.6) were used to review the manuscript for typographical and mathematical errors. The authors remain solely responsible for the mathematical content of the work.

\section{Connections and Banach Lie algebroids}\label{sect:fundamentals}
In this section we define basic concepts and recall some definitions. If nothing else is said, the Lie algebroids we consider (as well as all manifolds) are modelled on possibly infinite-dimensional Banach spaces. A basic reference for this setting is \cite{Lang01}.

\begin{setup}[Convention and notation]
If nothing else is said, manifolds are assumed to be smooth and modelled on Banach spaces. However, we do not enforce topological properties such as paracompactness by default. 

We write $\mathfrak{X}(M)$ for the (smooth) vector fields on a manifold and $\Gamma(E)$ for the smooth sections of a vector bundle $\pi \colon E \rightarrow M$.

For topological vector spaces $E,F$ we let $L(E,F)$ be the space of continuous linear operators. If $E,F$ are Banach spaces, we shall endow $L(E,F)$ with the operator norm topology.
\end{setup}

\subsection{Local coordinates in a Banach bundle}\label{loctriv}
Consider a Banach bundle $\pi\colon E\to M$ with typical fiber $\E$ over a Banach manifold  modelled on $\M$. For any open set $U$ in $M$, we denote simply by $E_{U}$  the restriction 
of $E$   over $U$ and  by $\G(E)$ (resp.  $\G(E_{U})$) the space of global  sections (resp.  sections defined on  $U$). If $\mathcal{F}(U)$ is the ring of smooth functions on  $U$ then $\G(E_{U})$ is an $\mathcal{F}(U)$-module. \\

Fix $x\in M$  and consider any chart $(U,\phi)$  around  $x$  such that  $E_{U}$ is isomorphic to the trivial bundle $U\times\E$. Modulo the identification of $U\subset M$ and $\phi(U)\subset
\mathbb{M}$ 
 we will write $E_{U}\equiv U\times \E$. Then we have
\[ TU = TM_{U}\equiv U\times \M \qquad TE_{U}\equiv U\times\E\times\M\times \E.\]
 Taking into account these equivalences, we get the following coordinates:
\begin{align*}(x,y) \text{ on }TM_{U}\equiv U\times \M, \quad\mathfrak{u}\equiv(x,\mathsf{u}) \text{ on }E_{U}\equiv U\times \E \\ (\mathfrak{u},\dot{ \mathfrak{u}})\equiv (x,\mathsf{u},y,v) \text{ on } TE_{U}\equiv U\times\E\times\M\times \E.\end{align*}



\subsection{Connections on a Banach bundle}
\label{banachconnection}
Consider a Banach bundle $\pi\colon E\to M$. 
The kernel of $T\pi\colon TE\longrightarrow TM$ is denoted by $VE$ and is called the
\textit{vertical bundle} over $E$. It appears as a vector bundle over $E$. According to our convention we also write $VE_{U}\equiv(U\times\mathbb{E})\times\mathbb{E}$.

The vertical lift gives a canonical vector-bundle isomorphism of $VE$ with the pull-back:%
\[%
\begin{array}
[c]{ccc}%
E\times_{M}E\simeq\pi^{\ast}E & \overset{\widetilde{\pi}}{\longrightarrow} & E\\
\downarrow &  & \downarrow\pi\\
E & \overset{\pi}{\longrightarrow} & M
\end{array}
\]
where $\pi^{\ast}E$ is the pull-back of $\pi\colon E\to M$ over $\pi$.
More precisely, a canonical isomorphism $E\times_{M}E\rightarrow VE$ called the
\textit{vertical lift} $\mathrm{vl}_{E}$ defined by%
\[
\mathrm{vl}_{E}\left(  x,u,v\right)  =\overset{.}{\gamma}\left(  0\right)
\]
where $\gamma(t)=u+tv$. This map is fiber linear over $M$. Finally, we note that by identifying $M$ with the zero-section in $E$, tangent vectors to $E$ at $M$ split into vertical tangent vectors and vectors tangent to $M$, cf. \cite{Vil67}. 
\begin{defi}
\label{D_Connection} Let $J\colon VE\rightarrow TE$ be the canonical inclusion, then a \emph{(non linear) connection} on $E$ is a bundle morphism
$V\colon TE\rightarrow VE$ such that $V\circ J=\operatorname{id}_{VE}$.
\end{defi}

The datum of a connection $V$ on $E$ is equivalent to the existence of a
decomposition $TE=HE\oplus VE$ for the Banach bundle $E$ with $HE=\ker V$.
We then have the following diagram:%
\[%
\begin{array}
[c]{ccc}%
VE & \overset{V}{\longleftarrow} & TE\\
\mathrm{vl}_{E}\uparrow &  & \downarrow K\\
\pi^{\ast}E & \overset{\widehat{\pi}}{\longrightarrow} & E
\end{array}
\]
\begin{defi}
The smooth bundle morphism $K:=\hat{\pi}\circ \mathrm{vl}_{E}^{-1}\circ V\colon TE\rightarrow E$ over $\pi$ is called the \textit{connection map} or \textit{connector} of $V$. 
If moreover, $K$ is linear on each fiber, the connection $V$ is called a \textit{linear connection}.
\end{defi}
In each fiber $T_{(x,u)}E$, the kernel of $K$ is exactly the subspace $H_{\left(  x,u\right)  }E$ of $HE$ in $T_{(x,u)}E$. Therefore, the datum $K$ is equivalent to the datum $V$.

\begin{defi}
The \emph{horizontal bundle} associated to the connection $V$ with connector $K$ is $HE:=\ker K$. By construction $T_e\pi\colon  H_eE\to T_{\pi(e)}M$ is a linear isomorphism. Denote its inverse by $\operatorname{Hor}^E_e\colon T_{\pi(e)}M\to H_eE$. 
The map 
\[ \operatorname{Hor}^E\colon E\times_M TM \to TE, \qquad (e,X)\mapsto \operatorname{Hor}^E_e(X), \]
is called the \emph{horizontal lift} associated to the connection.
\end{defi}
The horizontal bundle $HE$ of a linear connection  on $E$ is always tangent to the zero section of $p_E\colon TE\to E$.
By \cite[Lemma 1]{Vil67} there exists a smooth map $\omega \colon U\times\mathbb{E} \rightarrow L(\mathbb{M},\mathbb{E})$, i.e. $\omega\left(  x,u\right)  \in L\left(  T_{x}M,T_{\left(  x,u\right)}E\right)$ for the connection $V$ such that
\begin{align*}
V(x,u,y,v)    =(x,u,0,v+\omega(x,u).y)\qquad
K(x,u,y,v)    =(x,v+\omega(x,u).y)
\end{align*}
where the lower dot is application of a linear map. The connection $V$ is linear if and only if $\omega$ is linear in the second variable. Write $L^2(\mathbb{E},\mathbb{M},\mathbb{E})$ for the space of continuous bilinear maps. As $\mathcal{C}_x(u,y)=\omega(x,u).y$ we then obtain a smooth 
\[\mathcal{C}\colon U \rightarrow L^{2}(\mathbb{E},\mathbb{M};\mathbb{E}),\quad x\mapsto \mathcal{C}_x
\]
called \textit{local Christoffel components} of the connection $V$, cf. \cite{Lang01}.

\begin{rem}
\label{R_ExistsConnection} \normalfont It is classical that if $M$ is smoothly regular\footnote{i.e. the manifold is regular as a topological space and admits smooth bump functions, cf. \Cref{app:smooth_para}.}  and  paracompact, then there always exists a connection on any Banach bundle $\pi\colon E\to M$. However, these
assumptions impose the same assumptions on the Banach space $\mathbb{M}$ and so this condition is somewhat restrictive.
\newline
On the other hand, it is well known that there exist linear connections on a Banach manifold without such assumptions.
For instance, the well known situations include the trivial connection if the tangent bundle is trivial $TM\equiv M\times\mathbb{M}$ is trivial, or linear connections on loop spaces, \cite{CrFe11}), Levi-Civita connections \cite{Sch23}.

If there exists a  (linear) connection on a  Banach bundle $\pi\colon E\to M$ then for any smooth map $f\colon N\to M$ then we have an unique induced (linear) connection on the pull-back $f^*E\to N$  of $E\to M$ (cf. \cite[Section 37.28]{KrMi97}). In particular, if there exists a linear connection on $TM$, then we have the same property for any submanifold of $M$. Unfortunately, we have no example of a Banach bundle without a (linear) connection.
\end{rem}

Given  a linear connection $V$  on a Banach bundle $\pi\colon E\to M$, it naturally induces a Koszul connection, also called covariant derivative (cf. \cite{Lang01}) $$\nabla\colon \mathfrak{X}(M)\times\G(E)\rightarrow\G(E).$$ 
Since any (linear) connection induces naturally a (linear) connection on the restriction $E_{U}$ of $E$ to any open set $U$ of $M$, we also obtain a covariant derivative
$\nabla^U\colon \mathfrak{X}(U)\times\G(E_{U})\rightarrow\G(E_{U})$ with the correspondent properties for any function $f$ on $U$, $X\in\mathfrak{X}(U)$ and $\mathfrak{u}\in\G(E_{U})$.\\

\begin{hyp}\label{hyp:Koszul_1_jet} 
Throughout the paper, a connection means a linear bundle connection given by a smooth connector. When dealing with a Koszul connection (or covariant derivativ) we will always assume it comes from a linear connection. 
\end{hyp}
We use the above hypothesis to avoid extension questions for globally defined Koszul connections. (cf. \cite{CaPe12,BGJP19,BCP21},  \cite[VIII.2]{Lang01}): 
If  $\nabla$ were only defined on {\bf global} section of $E$  and in general  any local section of $E$  (resp. local vector field) on $M$ can not always extended to a global section of $E$ (resp. global vector field) on $M$, the  operator $\nabla $ can not  always induces a (local) operator $\nabla^U$ as previously. 
  This implies in particular\footnote{We assume that the reader is familiar with the language of jet-bundles. An overview for Banach manifolds can be found in \cite[Chapter 9]{MRaOD92}, whereas \cite[Chapter 41]{KrMi97} contains a discussion for even more general spaces.} that

\begin{prop}\label{localdef} For any covariant derivative $\nabla$ over a Banach bundle $\pi\colon E\to M$, for any $x\in M$,  the value $\nabla_X\mathfrak{u}(x)$ depends only of the value of $X$ at $x$ and the $1$-jet of $\mathfrak{u}$ at $x$.
\end{prop}

Given a Koszul connection $\nabla$ on $A$ and a $C^1$-curve
$\g\colon [0,1]\rightarrow M$, as in the finite dimensional case, we can associate a Koszul connection on the pull-back bundle $\g^*A$, see e.g. \cite[Section 4.3]{Sch23} for an overview. Further, one can define the notion of parallelism and a parallel transport along curves using Koszul connections, cf.\ \cite[VIII.3]{Lang01}, \cite[p. 5]{KS25}.

In a local trivialization $E_{U}\equiv U\times\mathbb{E}$, identify a local section $\mathfrak{u}$ with a map $\mathfrak{u} \colon U\rightarrow\mathbb{E}$. Then $\nabla$ has the local expression:
\[
\nabla_{X}\mathfrak{u}=d\mathfrak{u}(X)+\mathcal{C}(\mathfrak{u},X)
\]
 where $\mathcal{C}\colon U \rightarrow L^{2}(\mathbb{E},\mathbb{M}
;\mathbb{E}), x\mapsto \mathcal{C}_x$ is the \emph{local Christoffel component} of $\nabla$.

 \begin{rem} \label{nablaH} Let $\nabla $ be a Koszul connection on $E$ (induced by a linear connection). The fiber $(HE)_{(x,u)}$ of the associated horizontal bundle $HE$ is 
\begin{align*}
(HE)_{(x,u)}&=\{ T\mathfrak{u}(X)\;|\; \mathfrak{u} \in \Gamma(E),  \mathfrak{u}(x)=u ,\; X \in T_xM,  \nabla_X\mathfrak{u}=0\} \\ &\equiv\{(y,v)\in T_{(x,u)}E\;\;|\; v+\mathcal{C}_x(u,y)=0\}\end{align*}

\begin{enumerate}
\item  If $\varphi\in HE_{(x,u)}$ for any section $  \mathfrak{u}$ such that $ \mathfrak{u}(x)=u$ since $T\pi \circ T\mathfrak{u}=\mathrm{id}_{TM}$ it follows that $T\pi \circ T\mathfrak{u}(\varphi)=\varphi$. Therefore if  $X=T\pi(\varphi)$,
we have $\nabla_X\mathfrak{u}=0$.
\item For local sections $\mathfrak{u}$ and $\mathfrak{v}$ defined in an $x$-neighbourhood, we have a canonical vertical lift $\{\nabla_{\rho(\mathfrak{v})}\mathfrak{u}\}^V(x,u)$ in $T_{(x,u)}E$ of $\nabla_{\rho(\mathfrak{v})}\mathfrak{u}$. Then we get 
$$\mathrm{Hor}_x^E(\rho(\mathfrak{v}))=T_x\mathfrak{u}(\rho(\mathfrak{v}))-\{\nabla_{\rho(\mathfrak{v})}\mathfrak{u}\}^V(x,u)$$
\item Finally, if $E_{U}\equiv U\times\mathbb{E}$ and $E_{U^{\prime}}\equiv
U^{\prime}\times\mathbb{E}$ are local trivializations such that $U\cap
U^{\prime}\not =\emptyset$, then we have a smooth map $g\colon U\cap U^{\prime
}\rightarrow \mathrm{GL}(\mathbb{E})$ such that $\mathfrak{u}^\prime|_{U^{\prime}}=g. \mathfrak{u}|_{U}$ for
any section defined on $U\cap U^{\prime}$. Therefore the Christoffel component
$\mathcal{C}$ and $\mathcal{C}^{\prime}$ of $\nabla$ on $U\cap U^{\prime}$ are linked by
the relation
\[
\mathcal{C}^{\prime}(\mathfrak{u}^\prime,X)=-dg(X;g^{-1}.\mathfrak{u}^\prime)+g\mathcal{C}(g^{-1}\mathfrak{u}^\prime,X).
\]
\end{enumerate}
\end{rem}
The following splitting is standard, see \cite{Vil67}. Its proof is the obvious modification  of the one in \cite[X, Section 4, Theorem 4.3]{Lang01}.
\begin{lem}[Tensorial splitting]\label{lem:tensor_splitting}
Let $K\colon TE\rightarrow E$ be the connector of a linear connection on $E\to M$. Then the map 
\[ \mathcal{D}_E\colon TE\to\pi^*TM\oplus\pi^*E, \qquad \mathcal D_E(\xi_e) = \bigl( e, T_e\pi(\xi_e), K(\xi_e) \bigr) \]
is a bundle isomorphism over $E$ with inverse $\mathcal{D}_E^{-1}(e,X,a) = \operatorname{Hor}^E_e(X) + \operatorname{vl}_e(a)$.
\end{lem}
\begin{proof}
Since $K\circ\operatorname{vl}_e = \operatorname{id}_{E_{\pi(e)}}$, and $T_e\pi|_{H_eE}\colon H_eE\to T_{\pi(e)}M$ is an isomorphism,
\[ T_e\pi(\operatorname{Hor}^E_e(X))=X, \quad K(\operatorname{Hor}^E_e(X))=0, \text{ and }  T_e\pi(\operatorname{vl}_e(a))=0, \quad K(\operatorname{vl}_e(a))=a. \] 
Hence  $\mathcal D_E \bigl( \operatorname{Hor}^E_e(X) + \operatorname{vl}_e(a) \bigr) = (e,X,a)$. 
Conversely, for $\xi_e\in T_eE$, $\xi_e-\operatorname{vl}_e(K(\xi_e)) \in H_eE$ and it projects to $T_e\pi(\xi_e)$. Therefore 
\[ \operatorname{Hor}^E_e(T_e\pi(\xi_e)) = \xi_e-\operatorname{vl}_e(K(\xi_e)). \] 
It follows that $\mathcal D_E^{-1}\mathcal D_E(\xi_e) = \xi_e$. Thus \(\mathcal D_E\) is a vector-bundle isomorphism.
\end{proof}
\subsection{Anchored bundles and connections}
Let us first recall the notion of an anchored bundle.
\begin{defi}
\label{D_AnchoredBanachBundle}${}$ Let $\pi \colon A\to M$ be a Banach bundle with typical fiber $\A$.
A morphism of vector bundles $\rho\colon A\rightarrow TM$ is
called an \textit{anchor} and    $\left(A,\pi,M,\rho\right)$ is  called a \emph{Banach
anchored bundle}.
\end{defi}

Given an anchor  $\rho\colon A\rightarrow TM$, for any open $U$ in $M$, this morphism induces a map, again denoted by $\r \colon  \G(A_{U})\rightarrow\G(TM_{U}) $ defined  for any $x\in M$ and any section $\mathfrak{u}$ of $A$ by: $\r(\mathfrak{u})(x)=\r\circ \mathfrak{u}(x)$.

\begin{setup}
From now   $E$ will denote any vector bundle without any additional structure, a priori,  (such that an anchor for instance),  while $A$ will denote a Banach bundle provided with at least  an anchor  $\rho$.
\end{setup}

\begin{nota}\label{loctho}\normalfont
In a local trivialization (see \Cref{loctriv}), we have:
\begin{eqnarray}\label{loctrivr}
\r(x,a)\equiv (x,a) \mapsto (x,r_x(a))
\end{eqnarray}
where $ x\mapsto r_x$ is a smooth map from  $U$ to $ L(\A,\M)$.
\end{nota}
We are now interested in several special types of curves.
\begin{defi}\label{D_admissible}
A curve $c\colon I\subset\mathbb{R}\rightarrow A$ of class $C^k$ ($k\geq 1$) is called \emph{admissible} if the
tangent vector $\dot{\gamma}(t)$ of $\gamma=\pi\circ c \colon U \rightarrow M$ is precisely
$\rho(c(t))$.
\end{defi}

\begin{defi} If $c\colon I\to A$ is an admissible curve and  $\g=\pi\circ c$ then $c$ a section along $\gamma$ and $c$ is called a \emph{geodesic} of $\nabla$ if $\nabla
_{\dot{\gamma}}c=\nabla^{\g}c\equiv0$. 
\end{defi}

In a local trivialization $A_{U}\equiv U\times\mathbb{A}$, an admissible curve
$c\colon I\rightarrow A_{U}$ is a geodesic of $\nabla$ if and only if, in local coordinates, $ c \equiv(x,u)$ is
a solution of the following system of differential equations:%
\[
\dot{x}=r_x(u), \quad
\dot{u}=-\mathcal{C}_x(u,\dot{x})
\]
where $\mathcal{C}$ is the local Christoffel component of $\nabla$ on $A_{U}$.
Using the usual solution theory for ordinary differential equations in Banach manifolds \cite{Lang01}, one can prove that a Koszul connection induces an exponential map 
$$\mathrm{Exp}^\nabla \colon \mathcal{U} \subseteq A \rightarrow M,\quad  \mathrm{Exp}^\nabla (u) =\Fl_\nabla^1 (u),$$
where $\Fl_\nabla^1(u)$ is the time one-evolution map for the geodesic with initial value $u$. It is defined on an open neighborhood $\mathcal{U}$ of the zero section. In general, the exponential is not a local-diffeomorphism.

For a connection on $TM$, the corresponding exponential map restricts, after shrinking its domain, to a fibre-wise local diffeomorphism onto an open set around the footpoint of the fibre. To see this, set $G(x,u)=-\dfrac{1}{2}\mathcal{C}_x(u,\rho(x)(u))$. We get a
vector field $S_{U}$ on $A_{U}$ which is homogeneous of degree $2$ in $u$ and one can prove that it defines a spray on $A_{U}$. Now, according to the compatibility conditions between the local Christoffel components \Cref{nablaH} 4., we obtain 
\begin{lem}\label{lem:Spray_from_Koszul}
Under \Cref{hyp:Koszul_1_jet}, every linear connection on $TM$ for a Banach manifold $M$ induces a unique global spray $\mathbf{S}$ associated to $\nabla$ such that $\mathrm{Exp}^\nabla = \exp^{\mathbf{S}}$ is the exponential map associated to the spray.
\end{lem}
Conversely, in the case of $A=TM$ (cf. \cite{Vil67}), given a spray $\mathbf{S}$ on $A$, we can construct a
unique connection $\nabla$ whose associated spray is $\mathbf{S}$ (e.g. \cite[Proposition 4.23]{Sch23}).

\begin{lem} \label{lem:con_total_space} Let $\pi\colon E\to M$ be a Banach vector bundle. If $E\to M$ and $p_M \colon TM\to M$ admit linear connections, then the tangent bundle $p_E\colon TE\to E$ admits a linear connection. 
\end{lem}

\begin{proof} 
Let $\mathcal{D}_E\colon TE\to\pi^{*}TM\oplus\pi^{*}E$ be the tensorial splitting associated to the connector $K$, cf. \Cref{lem:tensor_splitting}.
We first construct linear connections on the pullback bundles $\pi^{*}TM\to E$ and $\pi^{*}E\to E$. For a vector bundle $F$ over $M$, let $\operatorname{Hor}^{F}$ denote the horizontal lift associated to the given connection on $q \colon F\to M$. 
For $(e,f)\in\pi^{*}F$ and a tangent vector $\xi_e\in T_eE$, define  $\operatorname{Hor}^{\pi^{*}F}_{(e,f)}(\xi_e) := \Bigl( \xi_e, \operatorname{Hor}^{F}_{f} \bigl( T_e\pi(\xi_e) \bigr) \Bigr)$. 
Since $T_f q \Bigl( \operatorname{Hor}^{F}_{f} (T_e\pi(\xi_e)) \Bigr) = T_e\pi(\xi_e)$, this vector is tangent to the pullback bundle \(\pi^{*}F\). 
The resulting horizontal distribution depends smoothly on $(e,f)$, Since $\operatorname{Hor}^F_f$ depends linearly on the base vector, the resulting horizontal distribution on $\pi^*F$ is fibrewise linear and  complementary to the vertical bundle. Hence it defines a linear connection. We conclude that $\pi^{*}TM\to E$ and $\pi^{*}E\to E$ carry linear connections. Taking the Whitney sum of these horizontal distributions yields a linear connection on 
$ \pi^{*}TM\oplus\pi^{*}E \longrightarrow E$. 
Let \[ K^{\oplus}\colon T(\pi^{*}TM\oplus\pi^{*}E) \to \pi^{*}TM\oplus\pi^{*}E \]
be its connector. Transport it through the vector-bundle isomorphism $\mathcal{D}_E$ to define 
$K^{TE} := \mathcal{D}_E^{-1} \circ K^{\oplus} \circ T\mathcal{D}_E$. 
Since \(K^{\oplus}\) is a linear connector and $\mathcal{D}_E$ is a vector-bundle isomorphism, $K^{TE}$ is the connector of a linear connection on $p_E\colon TE\to E$.
\end{proof}

Local additions are standard tools in the construction of manifolds of mappings. See e.g. \cite{KrMi97,KS25,Sch23}.

\begin{defi}\label{localadd}${}$
Let $N$ be a Banach manifold which we identify with the zero section in $p_N\colon TN\to N$. A local addition on $N$ is a smooth map $\Sigma\colon \mathcal{U}\subset TN\to N$  where $\mathcal{U}$ is an open neighbourhood of $N\subset TN$ such that:
\begin{itemize}
\item[1.] the map $\theta = (p_N, \Sigma) \colon  TN \to N\times N$ is a diffeomorphism from  $\mathcal{U}$ onto an open neighbourhood $\mathcal{V}$ of the diagonal;
\item[2.]  $\Sigma(x,0)=x$ for all $x\in N$.
\end{itemize}
\end{defi}
The following statement is based on an application of  \cite[Lemma 3.15]{KS25} which makes it necessary to require paracompactness. Unfortunately, we do not know a way around this requirement. 

\begin{prop}\label{adapcon} Let $\pi\colon E\to M$ a Banach bundle and $M$ be paracompact. We assume that there exists a linear connection on $E$ and on $TM$. Then a local addition  $\Sigma\colon \mathcal{U}\subset TE\to E$ on $E$ is obtained from the spray of the connection.
\end{prop}
By \Cref{lem:reg_BB} it is equivalent to require paracompactness of the total space $E$ to the requirement that the base $M$ is paracompact.
\begin{proof}[{Proof of \Cref{adapcon}}] By \Cref{lem:con_total_space}, the Banach manifold $E$ admits a linear connection. Its geodesic spray has an exponential map 
$\operatorname{Exp}^E\colon \Omega\subseteq TE\to E$.  
The fibre derivative of \(\operatorname{Exp}^E\) along the zero section is the identity. Hence, by \cite[Lemma~3.15]{KS25}, after shrinking $\Omega$, $(p_E,\operatorname{Exp}^E)\colon \Omega\to E\times E$ is an open embedding onto an open neighbourhood of the diagonal. Thus $\operatorname{Exp}^E|_\Omega$ is a local addition.
\end{proof}

\subsection{Banach Lie algebroids and their morphisms}\label{Liealgebroid}

We endow the anchored bundles of the previous section with further structures to arrive at the notion of a Banach algebroid.

\begin{defi}
\label{D_AlmostLieBracketOnAnAnchoredBundle}
An \emph{almost Lie bracket} on an anchored bundle $\left(A,\pi, M,\rho\right)$ is a sheaf of antisymmetric bilinear maps
\[
\LB_{A_U} \colon \Gamma\left( A_{U}\right)  \times\Gamma\left(  A_{U}\right)
\to\Gamma\left(  A_U\right), U\subseteq M \text{open}
\]
which satisfies the following properties
\begin{enumerate}[label=(\textbf{AL \arabic*})]
\item \label{AL1} the Leibniz identity:
\[
\forall\left(   \mathfrak{u}_{1}, \mathfrak{u}_{2}\right)  \in \Gamma \left(A_{U} \right)  ^{2}, \forall f \in C^{\infty}(U)
,\ \LB[\mathfrak{u}_{1},f \mathfrak{u}_{2}]_{A_{U}}=f.\LB[\mathfrak{u}_{1}, \mathfrak{u}_{2}]_{A_{U}}+df(\rho( \mathfrak{u}_{1})). \mathfrak{u}_{2}.
\]
\item \label{AL2}For any  open set $U\subseteq M$ and any $(\mathfrak{u}_1,\mathfrak{u}_2) \in \Gamma(A_{U})^2$, the maps
\[
\mathfrak{u}_1,\mapsto \LB[\mathfrak{u}_1,\mathfrak{u}_2]_{A_{U}} \textrm{ and } \mathfrak{u}_2,\mapsto \LB[\mathfrak{u}_1,\mathfrak{u}_2]_{A_{U}}
\]
only depend on the $1$-jets of $\mathfrak{u}_1$ and $\mathfrak{u}_2$, respectively.
\end{enumerate}
By abuse of notation, such a sheaf of almost Lie brackets will be denoted $\LB_A$. The quintuple $(A,\pi, M,\rho,\LB_A)$ is called an \emph{almost algebroid}.
\end{defi}

\begin{rem}\label{AL2nottrue}\normalfont In the infinite-dimensional Banach setting there exist (almost) Lie brackets which satisfy \ref{AL1} but for which \ref{AL2} is replaced by 
the assumption that the maps in \ref{AL2} 
depend on some $k$-jets of $\mathfrak{u}_1$ and $\mathfrak{u}_2$ respectively with $k>1$
(See \cite{BGT18}). In finite dimension this situation does not occur (cf. \cite{Mar02}).
\end{rem}

\begin{nota}\label{locbrac}\normalfont
In the context of local trivializations (see \Cref{loctriv}), we represent $\rho$ by the map $\r$ as in \Cref{loctho} and note that 
each section $\mathfrak{u}\equiv\mathsf{u}$   and $\mathfrak{v}\equiv\mathsf{v}$ of  ${A}_U\equiv U\times\mathbb{A}$ identifies with a smooth map $\mathsf{u}, \mathsf{v} \colon U\to \mathbb{A}$ such that 
\begin{eqnarray}\label{loctrivrbracket}
\LB[\mathfrak{u},\mathfrak{v}]_A(x)\equiv d_x\mathsf{v}(r_x(\mathsf{u}(x))) -d\mathsf{u}(r_x(\mathsf{v}(x)))+\mathsf{C}_x(\mathsf{u}(x),\mathsf{v}(x))
\end{eqnarray}
where $ x\mapsto \mathsf{C}_x$ is a smooth map from  $U$ to the Banach space  $ L_a^2(\A,\A)$ of continuous skew symmetric bilinear maps on $\A$ with values in $\A$.
For $A=TM$, $r=\operatorname{id}_{TM}$ and $C=0$, this formula gives the usual Lie bracket of vector fields.
\end{nota}

\begin{defi}
\label{D_AlmostLieAlgebroid}
Let  $ \left(A,\pi,M,\rho,\LB_{A} \right) $  an Banach almost Lie algebroid.
 The \emph{Jacobiator}\index{Jacobiator}  is the $\mathbb{R}$-trilinear
map $J_{A_U}:\Gamma(A_{U})^{3}\to\Gamma(A_{ U})$ defined, for any open set $U$ in $M$ and any $\left( \mathfrak{u}_{1}, \mathfrak{u}_{2}, \mathfrak{u}_{3} \right) \in \Gamma(A_{U})^3$ by
\[
J_{A_{U}}( \mathfrak{u}_{1}, \mathfrak{u}_{2}, \mathfrak{u}_{3}
)=\LB[ \mathfrak{u}_{1},\LB[ \mathfrak{u}_{2}, \mathfrak{u}_{3}]_A]_{A}+\LB[ \mathfrak{u}_{2},\LB[ \mathfrak{u}_{3}, \mathfrak{u}_{1}]_{A}]_{A}+\LB[ \mathfrak{u}_{3},\LB[ \mathfrak{u}_{1}, \mathfrak{u}_{2}]_{A}]_{A}
\]
\end{defi}

\begin{defi}
\label{D_ConvenientLieAlgebroid}
A \emph{Banach Lie algebroid}
 is a  Banach almost Lie algebroid $ \left(A,\pi,M,\rho ,\LB_{A} \right) $ such that the associated Jacobiator $J_{A_{U}}$ vanishes identically on each module $\Gamma(A_{U})$.
\end{defi}

\begin{prop}[{\cite{CaPe20,BCP21}}]
\label{P_EquivalenceMorphismJEtensor}
Consider a Banach almost Lie algebroid  $\left(A,\pi,M,\rho ,\LB_{A} \right) $ and denote for $U\subseteq M$ open  by $\LB$ the Lie bracket of vector fields on $U$.
\begin{enumerate}
\item For any open $U\subseteq M$ and any $ \left(  \mathfrak{u}_{1}, \mathfrak{u}_{2} \right) \in \Gamma\left( A_{ U} \right) ^2$, the map
\[
\left(  \mathfrak{u}_{1}, \mathfrak{u}_{2} \right)\mapsto \rho \left( \LB[ \mathfrak{u}_1, \mathfrak{u}_2]_A \right) -\LB[\rho( \mathfrak{u}_1), \rho( \mathfrak{u}_2)]
\]
only depends on the $1$-jet of $\rho$ and the values of $\mathfrak{u}_1$ and $\mathfrak{u}_2$ at $x\in U$.
\item
If the  Jacobiator $J_{A_{U}}$ vanishes identically,  then  we have:
\begin{equation}
\label{eq_rhoCompatible}
\forall \left(  \mathfrak{u}_{1}, \mathfrak{u}_{2} \right) \in \Gamma\left( A_{ U} \right) ^2,\; \rho \left( [ \mathfrak{u}_1, \mathfrak{u}_2]_A \right) =[\rho( \mathfrak{u}_1), \rho( \mathfrak{u}_2)].
\end{equation}
\item
If property \eqref{eq_rhoCompatible} is true, then $J_{A_{U}}\colon \Gamma\left( A_{ U} \right)^3 \rightarrow \Gamma \left( A_{ U} \right)$ is a smooth morphism which takes values in $\ker \rho$ over $U$.
\end{enumerate}
\end{prop}

\begin{cor}
\label{C_rhoLieAlgebraMorphism}
If $ \left(A,\pi,M,\rho,\LB_{A} \right) $ is a Banach Lie algebroid, then $\rho$ induces a morphism of Lie algebras from $\Gamma(A_{ U})$ into $\mathfrak{X}(U)$.
\end{cor}

\begin{defi}\label{D_Pre-LieAlgeboird} Let  $ \left(A,\pi,M,\rho,\LB_{A} \right) $  an Banach almost Lie algebroid. Then  $ \left(A,\pi,M,\rho,\LB_{A} \right) $ is called a pre-Lie algebroid if  $\rho$ induces a morphism of Lie algebras from $\Gamma(A_{ U})$ into $\mathfrak{X}(U)$ for any open set $U$ in $M$.
\end{defi}

\begin{rem} \normalfont Of course for $ \left(A,\pi,M,\rho,\LB_{A} \right) $  pre-Lie algebroid, in general each Jacobiator $J_{A_{U}}$ does not vanishes. But,  from Proposition \ref{P_EquivalenceMorphismJEtensor},  in this case  $J_{A_{U}}$ take values in $\ker \rho$ over $U$.
\end{rem}

Let $({A}, \pi, M,\rho, \LB_{{A}})$ and  $({A}', \pi', M',\rho', \LB_{\mathcal{A}'})$  two Banach Lie algebroids. 
Given any open set $U$ in $M$, the set  of sections $\Gamma(A_U)$ has a structure of  $\mathcal{F}(U)$ module. We introduce a few notations:
\begin{defi}\label{defi:bun_mor_sect} For a bundle morphism $\Psi\colon A\to A'$ between Banach Lie algebroids over $\psi\colon M\to M'$ denote by $\psi^\ast A'$ the pull-back of $A'$ over $M$,
\begin{itemize}
\item $\Gamma(\psi^\ast A'_U)$ the $\mathcal{F}(U)$-module of sections of $\psi^\ast A'_U$,
\item $\Gamma_\Psi(A_U)$, the $\mathcal{F}(U)$-sub-module defined as the finite $C^\infty(U)$-linear span of the set $\left\{ \mathfrak{u}'\circ \psi, \;\; \mathfrak{u}'\in  \Gamma(\psi^\ast A'_{U'})\right\}.$
\end{itemize}
\end{defi}
Note that each $\widetilde{\mathfrak{u}}\in \G_\Psi(A_U)$ can be  written 
\begin{equation}\label{eq_PSI-decomposition}
\widetilde{\mathfrak{u}}=\dis\sum_{i=1}^n \alpha_i.(\mathfrak{u}'_i\circ \psi)
\end{equation}
{\bf But this decomposition is not unique}.
When $A'$ is finite dimensional,  if  $A'_U$ is trivial,  then  $\Gamma(\widetilde{A'}_U)$  is a finite rank free module over $\mathcal{F}(U)$  and in fact we have $\G_\Psi(A_U)=\Gamma(\widetilde{A'}_U)$. But of course this is no longer  true in general and so   we have $\G_\Psi(A_U)\not=\Gamma(\widetilde{A'}_U)$ when $\mathbb{A}$ is infinite-dimensional.

Now as in finite dimension (cf. \cite{HiMa90}),  we introduce morphisms of Lie algebroids which lead to a category of (possibly infinite-dimenisonal) Banach Lie algebroids.

\begin{defi}
\label{D_LieMorphism}
The bundle morphism $\Psi\colon A \rightarrow A'$ is called a \emph{Lie algebroid  morphism} over $\psi \colon M \rightarrow M'$ if it fulfills the following conditions:
\begin{enumerate}[label=(\textbf{LM \arabic*})]
\item$\rho' \circ \Psi= T\psi \circ \rho$ \label{LM1}
\item\label{LM2}for any section $\mathfrak{u}$ and $\mathfrak{v}$ in $\G(A_U)$ with  a $\Psi$-decomposition:
 \[\Psi\circ \mathfrak{u}=\dis\sum_{i=1}^n \alpha_i.(\mathfrak{u}'_i\circ\psi) \text{ and }\Psi\circ \mathfrak{v}=\dis\sum_{j=1}^m \beta_j.(\mathfrak{v}'_j\circ\psi) \text{ in } \Gamma_\Psi(\mathcal{A}_U)\]
 the map $\Psi ([\mathfrak{u},\mathfrak{v}]_{A})$ is defined as
\begin{equation}\label{eq_LieMoprhism}
\dis\sum_{i=1}^n \sum_{j=1}^m\alpha_i.\beta_j.([\mathfrak{u}'_i,
\mathfrak{v}'_j]_{A'}\circ\psi)+\sum_{j=1}^m\rho(\mathfrak{u})(\beta_j).(\mathfrak{v}'_j\circ\psi)-\sum_{i=1}^n\rho(\mathfrak{v})
(\alpha_i).(\mathfrak{u}'_i\circ\psi)
\end{equation}
\end{enumerate}
\end{defi}

\begin{prop}[{cf. \cite[Lemma 1.4]{HiMa90}}]\label{unicityofLieMoprhism}
For any  $x\in U$, the value of $\Psi (\LB[\mathfrak{u},\mathfrak{v}]_{A})$ in \eqref{eq_LieMoprhism} at $\psi(x)$ depends only on the $1$-jet of $\Psi\circ\mathfrak{u}$ and $\Psi\circ\mathfrak{v}$ at $x$ and so is independent of the chosen $\Psi$-decompositions.
\end{prop}

Before establishing \Cref{unicityofLieMoprhism}, let us comment on morphisms.
 \begin{rem}\label{Psibracket}${}$\normalfont According to the identity \eqref{eq_LieMoprhism}, $\Psi (\LB[\mathfrak{u},\mathfrak{v}]_{A})$ belongs to $\Gamma_\Psi(A_U)$.
\begin{enumerate}
\item[1.] If $\Psi(\mathfrak{u})=\mathfrak{u}'\circ\psi$ and $\Psi(\mathfrak{v})=\mathfrak{v}'\circ\psi$ then we have
  \begin{equation}\label{relativesection}
\Psi(\LB[\mathfrak{u},\mathfrak{v}]_A)=\LB[\mathfrak{u}',\mathfrak{v}']_{A'}\circ \psi
  \end{equation}
\item[2.] Let $\Psi\colon A\to A'$ be  bundle morphism  over the identity of $M$. Then for any open set $U$ in $M$, the module $\G_\Psi(A_U)$ is the $\mathcal{F}(U)$ module $\G(A_U)$ and so $\Psi$ is a Lie algebroid morphism if and only if  the assumption \ref{LM1} of \Cref{D_LieMorphism} is satisfied and the condition \ref{LM2} is replaced by
\[ \Psi(\LB[\mathfrak{u},\mathfrak{v}]_A)=\LB[\Psi(\mathfrak{u}),\Psi(\mathfrak{v})]_{A'},\]
\noindent for all $\mathfrak{u}$ and $\mathfrak{v}$ in $\G(A_U)$ and all open set $U$ in $M$.
\item[3.]Consider two Lie algebroid morphisms $\Psi\colon A\to A'$ over $\psi\colon M\to M'$ and $\Psi'\colon A'\to A^{\prime\prime}$. over $\psi'\colon M'\to M^{\prime\prime}$, then by same arguments as in \cite[after the proof of Lemma 1.4]{HiMa90}, we obtain that $\Psi'\circ\Psi$ is a Lie algebroid morphism from $A$ to $A^{\prime\prime}$, over $\psi'\circ\psi$. Hence, one obtains a category of Banach Lie algebroids.
\end{enumerate}
\end{rem}

\begin{proof}[Proof of \Cref{unicityofLieMoprhism}]
Fix some $x\in U$. Property \ref{LM1} in the local trivialisation reads $T_x\psi(r_x)=r'_{\psi(x)}\circ\Psi_x$. According to Notation \ref{loctho} we identify  sections with their local representatives $\mathsf{u}$ and $\mathsf{v}$ and suppress the distinction in the notation  $\widetilde{\mathsf{u}}:= \Psi \circ \mathsf{u}$ and $ \widetilde{\mathsf{v}}:=\Psi \circ \mathsf{v}$. Using the decompositions $\widetilde{\mathsf{u}} = \sum_i \alpha_i(\mathsf{u}'_i\circ\psi),  \widetilde{\mathsf{v}} = \sum_j \beta_j(\mathsf{v}'_j\circ\psi)$, the chain rule gives 
\begin{align*}
d_x\widetilde{\mathsf{v}}(r_x(\mathsf{u})) &= \sum_j d_x\beta_j(r_x(\mathsf{u}))\,\mathsf{v}'_j(\psi(x)) + \beta_j(x)\, d_{\psi(x)}\mathsf{v}'_j \bigl(T_x\psi(r_x(\mathsf{u}))\bigr), \text{  and }\\
d_x\widetilde{\mathsf{u}} (r_x(\mathsf{v}))&=\displaystyle\sum_{i} d_x\alpha_i(r_x(\mathsf{v}))\mathsf{u}'_i(\psi(x))+ \alpha_i(x)d_{\psi(x)}\mathsf{u}'_i(T_x\psi(r_x(\mathsf{v}(x)))),\end{align*}
 and an analogous relation for the maps  $\mathsf{u}$ and $\mathsf{v}$. Hence, as in Notation \ref{locbrac} we can write $[\mathfrak{u}'_i,\mathfrak{v}'_j]_{A'}(\psi(x))$ as
 \begin{align*}
 &d_{\psi(x)}\mathsf{v}'_j(r_{\psi(x)}'(\mathsf{u}'_i(\psi(x)))))- 
d_{\psi(x)}\mathsf{u}'_i(r_{\psi(x)}'(\mathsf{v}'_j(\psi(x))))) \\ +& \mathsf{C}'_{\psi(x)}\left(\mathsf{u}'_i(\psi(x)),
 \mathsf{v}'_j(\psi(x))\right)
 \end{align*}
 This expression at $\psi(x)$ shows that  \eqref{eq_LieMoprhism} becomes  
 \[d_x\widetilde{\mathsf{v}}(r_{x}({\mathsf{u}(x)}))-d_x\widetilde{\mathsf{u}}(r_{x}(\mathsf{v}(x))) +\mathsf{C}'_{\psi(x)}\left(\widetilde{\mathsf{u}}(x),\widetilde{\mathsf{v}}(x)\right).\]
 This depends only on $\widetilde{\mathsf{u}}(x), \widetilde{\mathsf{v}}(x), d_x\widetilde{\mathsf{u}}, d_x\widetilde{\mathsf{v}}$, and therefore only on the 1-jets of $\Psi\circ \mathfrak{u}$ and $\Psi\circ \mathfrak{v}$ at $x$.
 \end{proof}

\section{Time-dependent sections and homotopy}\label{timedependenhomotopy}
In this section we develop a flow approach for time dependent section in an algebroid. The aim is to recover the description of homotopies for admissible curves in Lie algebroids developed in \cite{CF03} for finite-dimensional vector fields. On the way we will also revisit the theory of linear vector fields on algebroids (see \cite{Mac05} for the finite dimensional theory).

The following notion of mixed (finite-order) differentiability will be useful. It was developed in \cite{AaS15}, cf. also \cite{GaS22} for a more advanced version). As these sources use Bastiani smoothness, we mention first how this applies to our setting

\begin{defi}
A continuous mapping on an open subset of a Banach space is called Bastiani $C^k$ for $k\in \N_0\cup\{\infty\}$ if all it's  (iterated) directional derivatives exist and glue to continuous mappings on a suitable cartesian product of the open set and the ambient space. For a non-open subset with dense interior and locally convex (in the sense that every point admits a convex neighborhood in the set) we say a mapping is Bastiani $C^k$ if it is continuous and Bastiani $C^k$-on the interior of the set, such that the iterated direction derivatives extend continuously to the boundary.
\end{defi}
\begin{rem}\label{rem:Bastiani_vs_Frechet}
In general Bastiani $C^k$ is weaker than the usual Frech\'{e}t differentiability (briefly denoted by $FC^k$). However, as outlined in \cite[1.28]{Sch23} and the references therein, Bastiani $C^{k+1}$ implies the $FC^k$-property. In particular, for smooth maps, both concepts coincide. If the domain of the mapping is a subset of a finite dimensional space, then we have $C^k = FC^k$.
\end{rem}
Following up on the remark it is important to note that in the sequel whenever we use finite orders of differentiability, the parameters for the finite order differentiability will vary in intervals or finite cartesian products of intervals. Thus we may freely exchange Bastiani and Frech\'{e}t differentiability.
\begin{defi}
A mapping $f \colon U \times V \rightarrow F$ on a cartesian product of locally convex subsets of Banach spaces with dense interior is called $C^{r,s}$-map (with $r,s \in \N_0\cup\{\infty\})$ if it is continuous and for each fixed $y \in V$ the map $f(\cdot , y)$ is Bastiani $C^r$ and each of its iterated directional derivatives is Bastiani $C^s$ with respect to the parameter $y \in V$. 
\end{defi}
Since $C^{r,s}$-mappings satisfy variants of the chain rule, it makes sense to define them on manifolds and consider vector bundle sections of this mixed differentiability class. As we will use them later, we recall the following chain rules from \cite[Lemma 3.17, Lemma 3.18 and Lemma 3.19]{AaS15}:

\begin{setup}\label{Crs_chainrules}
We assume that $f$ is a $C^{r,s}$-map for $r,s\in \N_0\cup\{\infty\}$ and $k\geq r+s$:
\begin{enumerate}
\item if $g_1$ is $C^r$ and $g_2$ is $C^s$, then $f\circ (g_1 \times g_2)$ is $C^{r,s}$
\item If $h$ is $C^{k}$, then $h\circ f$ is $C^{r,s}$
\item If $g$ is a $C^{r,k}$-map such that the composition $g\circ (\text{pr}_1,f)(x,y)=g(x,f(x,y))$ makes sense, the map $g\circ (\text{pr}_1,f)$ is a $C^{r,s}$-map.
\end{enumerate}
\end{setup}

\subsection{Time-dependent sections}\label{Timedependent}
In this section we consider time dependent sections for a Banach bundle. For this let us fix some notation: 
\begin{defi}
We work over a Banach bundle $\pi\colon A\to M$. A \emph{time-dependent (local) section} of $A$ is a map 
$\mathfrak{u}\colon I\times U\to A_{U}$ where $I$ is an  interval of $\R$, $U \subseteq M$ open. Typically we will discuss time-dependent sections of class $C^{\ell,\infty}$, i.e. we have some regularity $\ell\in \N_0 \cup \{\infty\}$ in time and smoothness in $x$.

We say that a time dependent section $\mathfrak{u} \colon I \times M \rightarrow A$ \emph{extends a path $c\colon I\to A$}  if $\mathfrak{u}(t,\g(t))=c(t)$ for $\g (t) := \pi (c(t))$. 
\end{defi}

\begin{setup}[Flows of time dependent $C^{r,s}$-sections]\label{setup:flowsCrs}
For any (local) time-dependent vector field $X$ on $M$ (defined on an open set $I\times U$ of $I \times M$), the flow $\Fl_X^{t,s}$ of $X$ is defined for initial time $s$,
$$\dis\frac{\p}{\p t} \Fl_X^{t,s}(x)=X(t, \Fl_X^{t,s}(x)), \quad   \Fl_X^{s,s}(x)=x$$   
Flows have the following properties (the proofs are standard \cite{Lang01}, but for the non standard differentiability cf. \cite[Theorem C]{AaS15}):
\begin{itemize}
\item for any $(s,x)\in I\times U$ there exists $\eta>0$ and a neighbourhood $W$ of $x$ such that $\Fl_X^{t,s}$ is defined on $W$ for $t\in [s-\eta,s+\eta]$
\item if the map $(s,x)\mapsto X(s,x)$ is $C^{k,\infty}$, 
\begin{enumerate}
\item the map $((s,t),x)\mapsto \Fl_X^{t,s}(x)$ is $C^{k,\infty}$ and for $(s,x)$ fixed the map $t \mapsto \Fl_X^{t,s}(x)$ is $C^{k+1}$;
 \item $\Fl_X^{t,s}$ satisfies the relation $\Fl_X^{t,s}\circ \Fl_X^{s,r}=\Fl_X^{t,r}\;$ if  $\Fl^{t,s}$, $\Fl^{ s,r}$  and $\Fl^{t,r}$ are defined.
 \end{enumerate}
 \end{itemize}
 \end{setup}

\begin{rem}\label{TMcase} 
\normalfont In  the particular  case of the anchored bundle $(TM,M,\mathrm{id})$, a time-dependent section is a \emph{time-dependent vector field} and a vector field along a path $\g\colon [0,1]\to M$  is nothing but an admissible curve in $TM$.
\end{rem}  
  
We fix a conventions for mappings of several variables.

\begin{setup}
    Denote for mappings $\g \colon [0,1] \times [0,1] \rightarrow M$ by a subscript always the map defined via $\g_\varepsilon := \g (\varepsilon , \cdot)$ on occasion we will write $\g^t:=\g(\cdot,t)$.
\end{setup}
The next extension result allows us to work with flow domains extending slightly beyond the parameter interval. Its proof is an application of the Seeley-extension theorem which we defer to \Cref{app:proof}

\begin{prop}\label{timedepend} Fix a Banach manifold $M$ and $k \in \N \cup \{\infty\}, k \geq 2$. We consider a $C^k$-map ${\g}\colon [0,1]\times [0,1] \to  M,$ such that
$$\g(\varepsilon,0)=\g(0,0) \text{ and }\g(\varepsilon,1)=\g(0,1), \text{ for all }\varepsilon\in [0,1]$$
There exists $\eta>0$ and a $C^k$-extension $\hat{\g} \colon [0,1] \times [-\eta, 1+\eta] \rightarrow M$ such that $\hat{\g}(\varepsilon,t)=\g_\varepsilon (0)$ for $t\in ]-\eta,-\frac{\eta}{2}]$ and $\hat{\g} (\varepsilon,t)=\g_\varepsilon (1)$ for all $t\in  [1+\frac{\eta}{2},1+\eta]$.
\end{prop}

We will now establish lifts for the extensions constructed in \Cref{timedepend}. For this we need connections both on the tangent bundle and on the anchored bundle. We introduce the following assumption:

\begin{hyp}\label{H} 
 A Banach Lie algebroid  $(A,\pi,M,\rho, \LB_A)$  satisfies the assumption {\bf  (H)} if there exists a linear connection on $A$ and on $TM$ and $M$ is regular as a topological space.
\end{hyp}
The following technical result establishes Theorem A from the introduction.
\begin{prop}\label{timedependent2}
 Let $(A,\pi,M,\rho,\LB_A)$ be a Banach Lie algebroid which satisfies Assumption {\rm \textbf{(H)}}. For $k\geq 2$ fix a $C^k$-map $\g\colon [0,1]^2 \rightarrow M$ together with $J:=[-\eta,1+\eta]$ and an extension $\hat{\g}\colon [0,1] \times J \rightarrow M$ as in \Cref{timedepend}.
 \begin{enumerate}
 \item For $s \in [0,1]$ there is $\delta_s >0$ and an open neighborhood $\mathcal{W}_s\subseteq J \times M$ of the graph of $\hat{\g}_\varepsilon$ for every $\varepsilon\in I_s := [0,1] \cap [s-\delta_s,s+\delta_s]$.
 
 Then for every  $C^\ell$-map ${c}\colon [0,1]^2\to A$, $\ell \leq k$, with $\g=\pi\circ c$, there is $C^{\ell,\infty}$-map ${\mathfrak{u}^s} \colon I_s \times \mathcal{W}_s \rightarrow A$, such that $\pi \circ \mathfrak{u}^s(\varepsilon , t,y)=y$ and  $\mathfrak{u} (\varepsilon, t,\g_\varepsilon(t))=c_\varepsilon(t)$ for $(\varepsilon, t)\in I_s \times [0,1]$.
 
 Further, we can choose $\mathfrak{u}_\varepsilon$ constant on $\left(\left([-\eta, -\frac{\eta}{2}] \cup [1+\frac{\eta}{2},1+\eta]\right) \times M\right) \cap \mathcal{W}_s$. If $c_\varepsilon(0)=0$ (or $c_\varepsilon(1)=0$) for $\varepsilon\in I_s$, then also ${\mathfrak{u}_\varepsilon}(0)=0$ (or $
  {\mathfrak{u}_\varepsilon}(1)=0$).
 \item  If $M$ is smoothly paracompact, the section $\mathfrak{u}$ can be defined on $[0,1] \times J\times M$.
  \end{enumerate}
\end{prop}

\begin{proof}
Choose bundle trivializations $\Psi_0:A_{U_0}\to U_0\times\mathbb A, \Psi_1\colon A_{U_1}\to U_1\times\mathbb A$, around $\gamma(0,0)$ and $\gamma(0,1)$, respectively. By continuity of $\hat{\gamma}$ we can pick $\t_0>0$ such that $ \widehat\gamma \bigl( [0,1]\times[-\eta,\tau_0] \bigr) \subseteq U_0$.
In the trivialization $\Psi_0$, write $ \Psi_0(c(\varepsilon,t)) = \bigl( \gamma(\varepsilon,t), a_0(\varepsilon,t) \bigr)$, $(\varepsilon,t)\in[0,1]\times[0,\tau_0]$, where $a_0\colon [0,1]\times[0,\tau_0]\to\mathbb A$ is of class $C^\ell$. Now apply the argument\footnote{note: $c$ does not need to have constant in $\varepsilon$ endpoints, however, as the basepoints enable us to work in one trivialisation, the arguments apply with the obvious modifications} from the proof of \Cref{timedepend} to obtain an extension (after possibly shrinking $\tau$) $\widetilde a_0\colon [0,1]\times[-\eta_0,\tau_0]\to\mathbb A$ of $a_0$ and constant in $t$ for a neighborhood of $-\eta_0$ (where $0<\eta_0<\eta$. Now 
\[ \widehat c_0(\varepsilon,t) := \Psi_0^{-1} \bigl( \widehat\gamma(\varepsilon,t), \widehat a_0(\varepsilon,t) \bigr), \qquad (\varepsilon,t)\in [0,1]\times[-\eta_0,\tau_0]. \]
is a $C^\ell$-map which satisfies $\pi\circ\widehat c_0=\widehat\gamma$, and agrees with $c$ on $[0,1]\times[0,\tau_0]$. It is constant on a neighborhood of $-\eta$. The construction at the right endpoint is analogous. Gluing the two endpoints maps with $c$ and shrinking $\eta$ one more time if necessary the piecewise definition 
\[ \widehat c(\varepsilon,t) := \begin{cases} \widehat c_0(\varepsilon,t), & -\eta_0\leq t\leq\tau_0, \\ c(\varepsilon,t), & \tau_0\leq t\leq\tau_1, \\ \widehat c_1(\varepsilon,t), & \tau_1\leq t\leq1+\eta_1 \end{cases} \] 
is well defined and of class $C^\ell$. By construction we have then $\hat{c}(\varepsilon,t)=c(\varepsilon,0)$ for $t<-\eta/2$. A similar identity holds on the right boundary.
We may always assume that $\eta_1,\eta_0$ were smaller than $\eta$, but as the curve is constant on a small interval around the end-points, extending it by a constant we may assume that both extensions of $\g$ and $c$ exist on the same interval. For $J=[-\eta,1+\eta]$ set 
\[K:= \hat{\g}([0,1]\times J), \qquad K_s:=\hat{\g}(\{s\}\times J), s \in [0,1].\]

\paragraph{Step 1: Preparation around the compact image}
As $K$ is compact.  Applying \Cref{lem:paracpt-nbhd_cmpt.set} we may work in a paracompact neighborhood of the image and thus assume without loss of generality that $M$ is paracompact. Thus \Cref{adapcon} (or rather its proof) implies that $\Sigma :=  \text{Exp}^\nabla$ is a local addition on a suitable neighborhood $\Omega$ of the $0$-section in $TM$.  As in \cite{KS25} we may then assume that $\Omega \cap T_xM$ is star shaped for every $x$ and $\mathcal{O}:=(\pi_M, \text{Exp}^\nabla)(\Omega)$ is an open neighborhood of the diagonal in $M\times M$. Let $\Theta:=(\pi_M,\Sigma) \colon \Omega \rightarrow \mathcal{O}$ be the diffeomorphism induced by $\Sigma$.

\paragraph{Step 2: Preparation in a local chart.}
 Let $(W,\kappa)$ be a chart of $M$ such that $x_\star \in K_s \cap W$, $\kappa (W)=B_5^E (0)$ (the norm ball in the model space $E$ of radius $5$ and $\kappa(x_\star) =0$. Note that there exists an open neighborhood $\Omega_\kappa \subseteq T\kappa (TW)$ of the zero section on which the map $\theta_\kappa := \kappa \circ \text{Exp}^\nabla \circ T\kappa^{-1}|_{\Omega_\kappa}$ makes sense. Hence we obtain a local addition on $B_5^E(0)$ via 
\begin{align*}
\Sigma_\kappa := (\kappa \times \kappa) \circ \Sigma \circ T\kappa^{-1}|_{\Omega_\kappa} = (\text{pr}_1, \theta_\kappa) \colon \Omega_\kappa \rightarrow B_5^E(0)\times B_5^E(0)   
\end{align*}
We find $R>0$ and $S>0$ such that for $P:= B_R^E(0)$ we have $P\times B_S^E(0) \subseteq \Omega_\kappa$. Then apply the quantitative inverse function theorem \cite[Proposition 2.1]{Glo06} to $ \theta_\kappa|_{P\times B_S^E(0)} \colon P \times B_S^E(0) \rightarrow B_5^E(0)$ with $P$ interpreted as the parameter. We obtain $0<r_{x_\star},s_{x_\star}<\min\{R,S\}$ such that 
\[y+B_{r_{x_\star}}^E (0) \subseteq \theta_\kappa(\{y\} \times B_{s_{x_\star}}^E(0)),\quad \forall y \in B_{r_{x_\star}}(0).\]
We obtain thus for every $x_\star$  neighborhood $W_{x_\star}:=\kappa^{-1}(B_{r_{x_\star}}^E(0))$ such that $U_{x_\star}:= T\kappa^{-1}(B_{r_{x_\star}}^E (0) \times B_{s_{x_\star}}^E) \subseteq \Omega \cap TW_{x_\star}$ and for every pair $(x,y) \in W_{x_\star}^2$ there is a unique $\nabla$-geodesic from $x$ to $y$.

\paragraph{Step 3: A graph neighbourhood.}
Put $J_- :=[-\eta,-\eta/2]$ and $J_+:=[1+\eta/2,1+\eta]$. By construction of $\widehat\gamma$, there are $x_\pm\in M$ such that $\widehat\gamma_\varepsilon(t)=x_\pm $ for $(t,\varepsilon)\in J_\pm \times [0,1]$. Choose open neighbourhoods $P_\pm$ of $x_\pm$ such that $P_\pm\times P_\pm\subseteq\mathcal O$ together with relatively open intervals $\hat{J}_\pm^\circ \subseteq J\) containing $\overline{J}_\pm$, respectively, such that $\widehat\gamma_s(J_\pm^\circ)\subseteq P_\pm$

Let now $J_M \subseteq J$ be a compact interval such that $J=\hat{J}_-\cup J_M\cup \hat{J}_+$.
For every $x_\star\in K_s$, Step 2 provides open neighbourhoods $W_{x_\star}^0 \subseteq \overline{W_{x_\star}^0}\subseteq W_{x_\star}^1$ such and every pair of points in $W_{x_\star}^1$ is joined by the unique $\nabla$-geodesic determined by the local addition. Equivalently, $W_{x_\star}^1\times W_{x_\star}^1 \subseteq\mathcal O$. 
For every $t_0\in J_M$, choose $x_\star(t_0)\in K_s$ such that $\widehat\gamma_s(t_0) \in W_{x_\star(t_0)}^0$. By continuity of $\widehat\gamma_s$, there is a relatively open interval $t_0 \in L_{t_0}\subseteq J$ such that $\widehat\gamma_s(\overline{L_{t_0}}) \subseteq W_{x_\star(t_0)}^0$. Choose a finite cover $J_M=\bigcup_{i=1}^N L_i$ and set $W_i^0:=W_{x_\star(t_i)}^0, W_i^1:=W_{x_\star(t_i)}^1$. 

For every $i$, we have $\widehat\gamma_s(\overline{L_i}) \subseteq W_i^0 \subseteq W_i^1$. Then $\widehat\gamma^{-1}(W_i^1)$ is an open neighbourhood of $\{s\}\times\overline{L_i}$ in $[0,1]\times J$. Using compactness of $\overline{L_i}$ and Wallace Lemma \cite[3.2.10]{Eng89} there is a relatively open neighbourhood $I_{s,i}$ of $s$ such that $\widehat\gamma_\varepsilon(t)\in W_i^1$ for all $(\varepsilon,t)\in I_{s,i}\times\overline{L_i}$.
 Choose $\delta_s>0$ such that $I_s := [s-\delta_s,s+\delta_s]\cap[0,1] \subseteq \bigcap_{i=1}^N I_{s,i}$. After shrinking $\delta_s$, continuity of $\widehat\gamma$ also ensures $\widehat\gamma_\varepsilon(t)\in P_\pm$ for  $\varepsilon\in I_s,\ t\in J_\pm^\circ$, 
 Define 
 \[ \mathcal W_s := \hat{J}_-^\circ \times P_-\cup \bigcup_{i=1}^N L_i\times W_i^1 \cup \hat{J}^\circ_+ \times P_+ \subseteq J \times M. \]
 Then $\mathcal W_s$ is an open neighbourhood of the graphs of $\widehat\gamma_\varepsilon$ for all $\varepsilon \in I_s$. For every $(\varepsilon,t,y)\in I_s \times \mathcal W_s\), one has $\bigl( \widehat\gamma_\varepsilon(t),y \bigr)\in\mathcal O$. This follows from $P_\pm\times P_\pm\subseteq\mathcal O$ on the enpoint intervals and from $(W_i^1)^2\subseteq\mathcal O$ on each middle piece.
 
\paragraph{Step 4: Construction of the extension.} 
For $\varepsilon\in I_s, (t,y)\in\mathcal W_s$, by Step~3 $(\widehat\gamma_\varepsilon(t),y)\in \mathcal{O}$. Hence there is a unique vector $\xi_{\varepsilon,t,y} \in \Omega \cap T_{\widehat\gamma_\varepsilon(t)}M$ such that $\Theta(\xi_{\varepsilon,t,y}) = \bigl( \widehat\gamma_\varepsilon(t),y \bigr)$. Equivalently, $\operatorname{Exp}^{\nabla}_{\widehat\gamma_\varepsilon(t)} \bigl( \xi_{\varepsilon,t,y} \bigr) = y$. Since the fibres of $\Omega$ are star-shaped, the curve 
\[ \varsigma_{\varepsilon,t,y}\colon [0,1]\to M, \qquad \varsigma_{\varepsilon,t,y}(r) := \operatorname{Exp}^{\nabla}_{\widehat\gamma_\varepsilon(t)} \bigl( r\xi_{\varepsilon,t,y} \bigr), \] 
is well defined and the unique $\nabla$-geodesic with $\varsigma_{\varepsilon,t,y}(0) = \widehat\gamma_\varepsilon(t),\varsigma_{\varepsilon,t,y}(1)=y$. Let $\operatorname{Pt}^{A}_{\varsigma_{\varepsilon,t,y}} \colon A_{\widehat\gamma_\varepsilon(t)} \to A_y$ denote parallel transport along $\varsigma_{\varepsilon,t,y}$ with respect to the chosen linear connection on $A$. Define 
\[ \mathfrak{u}^s(\varepsilon,t,y) := \operatorname{Pt}^{A}_{\varsigma_{\varepsilon,t,y}} \bigl( \widehat c(\varepsilon,t) \bigr), \qquad \varepsilon\in I_s,\quad (t,y)\in\mathcal W_s.\]
Then $\pi\bigl(\mathfrak{u}^s(\varepsilon,t,y)\bigr)=y$, so, for every fixed $(\varepsilon,t)$, $y\mapsto \mathfrak{u}^s(\varepsilon,t,y)$ is a local section of $A$. 
The map $I_s\times \mathcal{W}_s \rightarrow \Omega,\quad (\varepsilon,t,y) \mapsto  \xi_{\varepsilon,t,y}= \Theta^{-1}\bigl( \widehat\gamma_\varepsilon(t),y \bigr)$
is of class $C^{\ell,\infty}$ (cf. Step 2 and \cite[Theorem 2.3 (c)]{Glo06}). We are a bit sloppy in the notation $C^{\ell,\infty}$ means with respect to the variable splitting $((t,s),x)$. Now parallel transport is smooth in its parameters. This follows from the fact that it is the solution of an ordinary differential equation whose right hand side is smooth in these parameters cf.\ \cite[Proposition 4.1 and its proof]{KS25}. We deduce that $\mathfrak{u}^s$ is a $C^{\ell,\infty}$-map on its domain.

If $y=\widehat\gamma_\varepsilon(t)$, then $\Theta^{-1} \bigl( \widehat\gamma_\varepsilon(t), \widehat\gamma_\varepsilon(t) \bigr) = 0_{\widehat\gamma_\varepsilon(t)}$. Hence $\varsigma_{\varepsilon,t, \widehat\gamma_\varepsilon(t)}$ is the constant curve at $\widehat\gamma_\varepsilon(t)$ and on this curve parallel transport is the identity. We deduce
$\mathfrak{u}^s \bigl( \varepsilon,t,\widehat\gamma_\varepsilon(t) \bigr) = \widehat c(\varepsilon,t)$. In particular, $\mathfrak{u}^s \bigl( \varepsilon,t,\gamma_\varepsilon(t) \bigr) = c(\varepsilon,t)$ for $(\varepsilon,t)\in I_s\times[0,1]$. 
For $t\in[-\eta,-\eta/2]$, the extensions satisfy \[ \widehat\gamma_\varepsilon(t)=\gamma_\varepsilon(0), \qquad \widehat c(\varepsilon,t)=c_\varepsilon(0). \]
Here we use that at the cylindrical inclusions $J_\pm\times P_\pm \subseteq \mathcal{W}_s$ ensure that the identities holds as equalities of sections on $P_\pm$.
Hence, whenever $(t,y)\in\mathcal W_s$, the vector $\xi_{\varepsilon,t,y} = \Theta^{-1} \bigl( \gamma_\varepsilon(0),y \bigr)$ and the corresponding geodesic are independent of $t$. It follows that $\mathfrak{u}^s(\varepsilon,t,y)$ is independent of $t$ on the left endpoint interval, wherever defined. The same argument applies on $[1+\eta/2,1+\eta]$. If $c_\varepsilon(0)=0_{\gamma_\varepsilon(0)}$, then linearity of parallel transport gives $\mathfrak{u}^s(\varepsilon,t,y)=0_y$ on the left endpoint interval. The corresponding statement holds at $t=1$.

\paragraph{Proof of part 2.}From the first part we obtain for every $s \in [0,1]$ an open set $\mathcal{W}_s \subseteq J \times M$ containing the graph of all $\hat{\g}_\varepsilon$ for $\varepsilon \in I_s$ and a $C^{\ell,\infty}$-section $\mathfrak{u}^s \colon I_s \times \mathcal{W}_s \rightarrow A$ such that $\mathfrak{u}^s \bigl( \varepsilon,t,\widehat\gamma_\varepsilon(t) \bigr) = \widehat c(\varepsilon,t)$ for the extension of the $C^\ell$ map $c \colon [0,1]^2 \rightarrow A$,

By compactness, there are finitely many $s_i \in [0,1], 1\leq i \leq n$ and $\overline{J}_i \subseteq I_{s_i}$ such that $[0,1] \subseteq \bigcup_{i=1}^N J_i$. Pick a partition of unity $\lambda_i \in C^\infty ([0,1],[0,1]), 1\leq i \leq N$ subordinate to the covering $\{J_i\}$ of $[0,1]$.

For every $i$, define the compact subset 
\[ K_i := \left\{ \bigl( t,\widehat\gamma_\varepsilon(t) \bigr): \varepsilon\in\operatorname{supp}\lambda_i,\; t\in J \right\} \subseteq J\times M. \]
By Part~1 we have $K_i\subseteq\mathcal W_{s_i}$.
Since $J\times M$ is smoothly normal (this follows e.g. from \cite[Corollary 16.17]{KrMi97} as both spaces are smoothly paracompact and metrisable), we can choose $\chi_i\in C^\infty(J\times M,[0,1]) $ such that $ \chi_i=1$ on a neighbourhood of $K_i$ and $\operatorname{supp}\chi_i \subseteq\mathcal W_{s_i}$.

For $\varepsilon\in I_{s_i}$, define 
\[ \overline{\mathfrak{u}}^i_\varepsilon(t,y) := \begin{cases} 
\chi_i(t,y)\, \mathfrak{u}^{s_i}(\varepsilon,t,y), & (t,y)\in\mathcal W_{s_i}, \\
0_y, & (t,y)\notin\mathcal W_{s_i}. 
\end{cases}\] 
Since $\operatorname{supp}\chi_i \subseteq\mathcal W_{s_i}$, the map is a well-defined $C^{\ell,\infty}$-section on $J\times M$. Now define 
\[
\widetilde{\mathfrak{u}}^{i}(\varepsilon,t,y) := \begin{cases} \lambda_i(\varepsilon)\, \overline{\mathfrak{u}}^{i}_\varepsilon(t,y), & \varepsilon\in I_{s_i}, \\ 0_y, & \varepsilon \notin I_{s_i}. \end{cases}
\] 
By construction of the partition of unity, the map is again $C^{\ell,\infty}$ across the boundary of $I_{s_i}$. Set 
$\hat{\mathfrak{u}}(\varepsilon,t,y) := \sum_{i=1}^N \widetilde{\mathfrak{u}}^{i}(\varepsilon,t,y)$. 
This is a finite sum of $C^{\ell,\infty}$-sections and therefore a $C^{\ell,\infty}$-section on $[0,1]\times J\times M$. If $\lambda_i(\varepsilon)\neq0$, then $\bigl( t,\widehat\gamma_\varepsilon(t) \bigr)\in K_i$, and hence, $\chi_i \bigl( t,\widehat\gamma_\varepsilon(t) \bigr) = 1$. Consequently, 
\[\hat{\mathfrak{u}} \bigl( \varepsilon,t,\widehat\gamma_\varepsilon(t) \bigr) = \sum_{i=1}^N \lambda_i(\varepsilon) u^{s_i} \bigl( \varepsilon,t,\widehat\gamma_\varepsilon(t) \bigr)  = \left( \sum_{i=1}^N\lambda_i(\varepsilon) \right) \widehat c(\varepsilon,t) = \widehat c(\varepsilon,t). \]
It remains to arrange the properties at the ends of the extended interval. 
\[\text{Choose } r\in C^\infty (J,[0,1]) \text{ such that } \begin{cases} r(t)=0  & t\in[-\eta,-\eta/2], \\ r(t)=t & t\in[0,1],\\ r(t)=1 &t\in[1+\eta/2,1+\eta].\end{cases}\]
Setting $\mathfrak{u}(\varepsilon,t,y) := \hat{\mathfrak{u}} \bigl( \varepsilon,r(t),y \bigr)$, we obtain a $C^{\ell,\infty}$-section.
Although the cutoffs $\chi_i$ may destroy the constancy in $t$ of the local extensions at the ends, the reparametrization $r$ restores this property without changing the family on $[0,1]$.
By construction $\mathfrak{u}$ lifts $c$ and for $t\in[-\eta,-\eta/2]$ one has with $\widehat\gamma_\varepsilon(t) = \gamma_\varepsilon(0), \widehat c(\varepsilon,t) = c(\varepsilon,0)$, i.e. $\mathfrak{u} \bigl( \varepsilon,t,\widehat\gamma_\varepsilon(t) \bigr) = c(\varepsilon,0) = \widehat c(\varepsilon,t)$. Moreover, the section is independent of $t$ on this interval. The same argument applies on $[1+\eta/2,1+\eta]$. Finally, if $c_\varepsilon(0)=0_{\gamma_\varepsilon(0)}$, then every local extension from Part~1 vanishes at $t=0$. Hence the same holds for $\mathfrak{u}$ on the boundary intervals. 
\end{proof}

\subsection{Linear vector fields on Banach bundles}\label{LinearVectorFields}
Following \cite{Mac05}  we  consider the notion of linear vector fields:

\begin{defi}\label{linE}  Let $(E,\pi,M)$ be a Banach bundle and $U$ be an open set of $M$. A \emph{linear local}  vector field on $E$ is a pair $(\mathcal{X},X)$ where $\mathcal{X}$ is a vector field on $E_{U}$ and $X$ a vector field on $U$ such that
\begin{equation}\label{XX}
\begin{tikzcd}
E_U \arrow[r,"\mathcal{X}"] \arrow[d,"\pi"]& TE_U \arrow[d,"T\pi"] \\ U \arrow[r,"X"] & TU
\end{tikzcd}
\end{equation}
is a bundle morphism. When $U=M$ the pair $(\mathcal{X},X)$ is simply called a \emph{linear vector field} on $E$.
\end{defi}

The set of linear vector fields defined over an open set $U$ is {\it stable by linear combinations}.
In a bundle trivialisation $TE_U \cong TU \times F \times F$ we write $\bullet_{T\pi}$ for vector operations in the $T\pi$-fibre. Then a linear vector field $(\mathcal{X},X)$ satisfies (cf. \cite[p.110 ff.]{Mac05}) 
\begin{align}\label{prolong_linear}\mathcal{X}(x,e_1+\lambda e_2, f_1+\lambda f_2)=\mathcal{X}(x,e_1,f_1) +_{T\pi} \lambda \cdot_{T\pi} \mathcal{X}(x,e_2,f_2).
\end{align}
In the following we will need the dual bundle of $A\rightarrow M$. We denote the dual bundle by $A^\ast \rightarrow M$ and will use that for local sections $\xi \in \Gamma(A_U)$ and $\sigma (A_U^\ast)$ there is a canonical pairing between $A$ and $A^\ast$ given by $$\langle \sigma (x),\xi(x)\rangle := \sigma (x)(\xi(x)), x\in U.$$
Note that this canonical pairing does not imply that there is a bundle isomorphism between $A$ and the dual bundle.

\begin{defi}
Let $\pi \colon E\to M$ be a Banach bundle and $\pi_{*}\colon E^*\to M$ it's associated dual Banach bundle. If $U$ is an open set of $M$, for any $ \s\in \G(A^*_{U})$ let $$\Phi^\s\colon E_{U}\to \R, \quad \Phi^\s(\mathfrak{u})=\langle\s\circ \pi(\mathfrak{u}), \mathfrak{u}\rangle$$ for any $\mathfrak{u}\in E_{U}$. Dually, for any $\mathfrak{u}\in \G(E_{U})$, we obtain a map $$\Phi_\mathfrak{u} \colon E_U^\ast \rightarrow \R, \quad  \Phi_\mathfrak{u}(\xi):=\langle\xi,\mathfrak{u}\circ \pi^\ast (\xi)\rangle.$$
Both $\Phi^\sigma$ and $\Phi_\mathfrak{u}$ are by construction fibre-wise linear.
 \end{defi}
\begin{setup}
We denote by $\mathcal{F}(E_{U})$ the ring of smooth functions on $E_{U}$. A function $f\in {\cal F}(E_{U})$ is called \emph{linear}, if the restriction of $f$ to each fiber $E_x=\pi^{-1}(x), x\in U$ is a linear map. Finally, we let $\mathcal{F}_L(E_{U})$ be the subset of linear  functions  of type $\Phi^\s,$ for all $\s\in \G(A^*_{U})$. We set $\mathcal{F}_\pi(E_{U})=\{f\circ \pi, f\in \mathcal{F}(U)\}$. Dually we have the same spaces for $E^\ast_U$, with the obvious modifications (i.e. $\mathcal{F}_L(E_{U}^\ast)$ is generated by the function $\Phi_\mathfrak{u}$ for $\mathfrak{u}\in \Gamma(E_U)$.
\end{setup}

\begin{prop}\label{chalinX}Let $\mathcal{X}$ be  a vector field on $E_{U}$.  The following properties are equivalent:
\begin{enumerate}
\item[{\rm (i)}] There exists a vector field $X$ on $U$ such that  $(\mathcal{X},X)$ is a linear  local vector field on $E$.
\item[{\rm (ii)}]  When we consider $\mathcal{X}$ as a map from $\mathcal{F}(E_{U})$ to $\mathcal{F}(E_{U})$ then $\mathcal{X}(\mathcal{F}_L(E_{U}))\subset \mathcal{F}_L(E_{U})$ and $\mathcal{X}(\mathcal{F}_\pi(E_{U}))\subset \mathcal{F}_\pi(E_{U})$.
\item[{\rm (iii)}] The local flow $\mathfrak{F}_{\mathcal{X}}^t$ of $\mathcal{X}$ is a local bundle isomorphism for $E$ over the flow $\Fl_X^t$ of  $X$.
\end{enumerate}
\end{prop}

\begin{proof}
Assume that (i) is true, we work in a local trivialisation $E_V \cong V \times \mathbb{E}$ over $V\subseteq U$ open. As we have a linear vector field, \eqref{prolong_linear} yields $\mathcal{X}(y,a)=(y,a,X(y),B(y).a)$ for some $y \mapsto B(y)\in L(\mathbb{A},\mathbb{A})$ smooth. 
Pick $F\in \mathcal{F}_\pi(E_{U})$ with $F=f\circ \pi$ on $V$, then $\mathcal{X}(f\circ \pi) (y,a)=d_x f(X(y))=X(f)\circ \pi $. In other words $\mathcal{X}(\mathcal{F}_\pi (E_U))\subseteq \mathcal{F}_\pi (E_U)$. 
If $F\in \mathcal{F}_L(E_{U})$, i.e. $F=\Phi^\sigma$ for some $\sigma$ (on $V$) we use that 
$d_{(y,a)}\Phi_\sigma(w,b) = \langle d_y\sigma(w),a\rangle + \langle\sigma(y),b\rangle$. 
So in this case 
\begin{align*}
\mathcal{X}(\Phi^\sigma)(y,a)&= d_{(y,a)}\Phi_\sigma \bigl( X(y),B(y).a \bigr) = \left\langle d_y\sigma\bigl(X(y)\bigr), a \right\rangle + \left\langle \sigma(y),B(y).a \right\rangle \\ &= \left\langle d_y\sigma\bigl(X(y)\bigr) + B(y)^T\sigma(y), a \right\rangle = \Phi^{\tilde \sigma} (y,a).
\end{align*}
where $B(y)^T$ is the transpose and $\tilde{\sigma}(y):= d_y\sigma\bigl(X(y)\bigr) + B(y)^T\sigma(y)$. Thus  $\mathcal{X}(\mathcal{F}_L(E_{U}))\subset \mathcal{F}_L(E_{U})$ and so $\textrm{(i)}\Rightarrow\textrm{(ii)}$.\smallskip

Assume that (ii) holds. Since $\mathcal{X}$ preserves $\mathcal{F}_\pi(E_U)$, it projects to a vector field $X$ on $U$(we can test locally with a point separating family of functionals, showing that $T\pi \mathcal{X} (e)$ only depends on $\pi(e)$, whence $X:=T\pi \circ \mathcal{X} \circ 0_E$ is the desired vector field). 
Choose a local trivialization $E_V\cong V\times\mathbb E$.  The local representative of $\mathcal{X}$ has the form $\mathcal{X}(y,a) = \bigl(y,a,X(y),\Xi(y,a)\bigr)$, 
where $\Xi\colon V\times\mathbb E\to\mathbb E$ is smooth. Let $\sigma\colon V\to\mathbb E^*$ be a local section of the dual bundle. Using the definition of $\Phi_\sigma$ we get the following identity for $X(\Phi_\sigma)(y,a)$
\begin{align} 
 d_{(y,a)}\Phi_\sigma \bigl(X(y),\Xi(y,a)\bigr) = \left\langle d_y\sigma\bigl(X(y)\bigr),a \right\rangle + \left\langle \sigma(y),\Xi(y,a) \right\rangle.\label{eq:loc_mX}\end{align}
By assumption, $\mathcal{X}(\Phi_\sigma)$ is fibrewise linear. The first term on the right-hand side of \eqref{eq:loc_mX} is linear in $a$, so $\left\langle\sigma(x),\Xi(x,a)\right\rangle$ is linear for every $\sigma(x)\in\mathbb E^*$. Since the continuous dual $\mathbb E^*$ separates points, $\Xi(x,a)$ must be linear and there is a smooth map $B\colon  V\to L(\mathbb E,\mathbb E)$, with $\Xi(x,a)=B(x).a$. Let $\eta(t)$ be an integral curve of $X$. Using the local representative of $\mathcal{X}$, the fibre component $a_\eta(t)$ of an integral curve of $\mathcal{X}$ over $\eta(t)$ satisfies 
$\dot a_\eta(t)=B(\eta(t)).a_\eta(t)$.  The associated evolution operator $P_{t,s}(x): \mathbb E\to\mathbb E$ is linear and satisfies 
\[ \frac{\partial}{\partial t} P^{t,s}(x) = B\bigl(\Fl^{X}_{t,s}(x)\bigr). P^{t,s}(x), \qquad P_{s,s}(x)=\operatorname{id}_{\mathbb E}. \] 
Its inverse is $P^{s,t}(\operatorname{Fl}_{X}^{t,s}(x))$. Therefore the local flow of $\mathcal{X}$ is given by $\Fl_{\mathcal{X}}^{t,s}(y,a) = \left( \Fl_{X}^{t,s}(y), P^{t,s}(y).a \right)$, and hence restricts to a continuous linear isomorphism $E_y \rightarrow E_{\Fl_{X}^{t,s}(y)}$ on every fibre on which it is defined. This proves (iii). 
\smallskip 

 Assume that $ \textrm{(iii)} $ is true. Clearly we have for the vector field associated to the flow that  $T\pi\circ\mathcal{X}=X\circ \pi$, the diagram \eqref{XX} is commutative and we must have $\pi\circ \mathfrak{F}_{\mathcal{X}}^t=\Fl_X^t\circ \pi$. 
 Let $a $ and $b$  in $E_U$ such that $\pi(a)=\pi(b)=x$.  Since $\mathfrak{F}_\mathcal{X}^t$ is an isomorphism from $E_x$ over $E_{\Fl_X^t(x)}$  we have $\mathfrak{F}_\mathcal{X}^t(a+b)=\mathfrak{F}_\mathcal{X}^t(a)+\mathfrak{F}_\mathcal{X}^t(b)$. 
Therefore we get: 
\[\mathcal{X}(a+b)=\left.\dis\frac{d}{dt}\right|_{t=0} \mathfrak{F}_\mathcal{X}^t(a+b)_=\left.\dis\frac{d}{dt}\right|_{t=0} \mathfrak{F}_\mathcal{X}^t(a)+\left.\dis\frac{d}{dt}\right|_{t=0} \mathfrak{F}_\mathcal{X}^t(b)=\mathcal{X}(a)+\mathcal{X}(b).\]
A similar argument yields $\mathcal{X}(\lambda a)=\lambda\cdot_{T\pi}\mathcal{X}( a), \lambda \in \R$. Thus $ \textrm{(iii)}\Rightarrow \textrm{(i)}$.
 \end{proof}

From Proposition \ref{chalinX} it follows that if $(\mathcal{X}_1, X)$ and $(\mathcal{X}_2, X_2)$ are local linear vector fields defined on an open set $U$  then {\it the Lie bracket $([\mathcal{X}_1,\mathcal{X}_2],[X_1, X_2])$ is also linear vector fields} defined on the open set $U$. 

For any $\s\in \G(A^*_{U})$ we denote again by $\Phi^\s\colon A_{U}\to \R, \Phi^\s(\mathfrak{u})=\langle\s\circ \pi(\mathfrak{u}), \mathfrak{u}\rangle$ for any $\mathfrak{u}\in A_{U}$ and define  
\begin{align}\label{Lsigma}{L}_\mathfrak{u}\s(\mathfrak{v}):=d(\s(\mathfrak{v}))(\rho(\mathfrak{u}))-\s([\mathfrak{u},\mathfrak{v}]_A)
\end{align}
This yields a well-defined local section of $A^*$. In a local trivialization with $r_x$ and $C_x$ representing the anchor and bracket, $\mathsf{v}$ a section, the product rule gives 
\begin{align*}
d_x(\sigma(\mathsf{v}))(r_x(\mathsf{u}(x))) =\left\langle d_x\sigma(r_x(\mathsf{u}(x))),\mathsf{v}(x) \right\rangle+ \left\langle \sigma(x),d_x\mathsf{v}(r_x(\mathsf{u}(x))) \right\rangle. \end{align*}
Inserting in \eqref{Lsigma} and using \eqref{loctrivrbracket}, the terms involving $d_x\mathsf{v}$ cancel, i.e. 
\begin{align*}
(L_{\mathsf{u}}\sigma)(\mathsf{v})(x) =& \left\langle d_x\sigma(r_x(\mathsf{u}(x))),\mathsf{v}(x) \right\rangle+ \left\langle \sigma(x), d_x\mathsf{u}(r_x(\mathsf{v}(x)))- C_x(\mathsf{u}(x),\mathsf{v}(x)) \right\rangle
\end{align*} 
Thus the right-hand side depends only on $\mathsf{v}(x)$, and it is continuous and linear, whence an element $A_x^*$ smoothly depending on $x$. Thus $\Phi^{{L}_\mathfrak{u}\s}\colon A_U\rightarrow \R$ is a well defined linear function. 

The following yields examples of (local) linear vector fields on $A$. It establishes the complete lift from Theorem B of the introduction. Also cf. \Cref{Xivalue}. 

\begin{theo} \label{thm:complete-lift} Let $(A,\pi,M,\rho,\LB_A)$ be a Banach almost Lie algebroid and $\mathfrak{u}\in\Gamma(A_U)$ a local section. Then there exists a unique linear vector field $(\mathcal{X}_\mathfrak{u}, \rho(\mathfrak{u}))$ on $A_U$, such that
\begin{enumerate} 
\item for every local section $\sigma$ of $A^*$ we have 
\begin{align}\label{Xs}
\LXu (\Phi^\sigma)=\Phi^{L_\mathfrak{u}\sigma}, \quad T\pi\circ \LXu = \rho(\mathfrak{u}) \circ \pi.
\end{align}
In local coordinates this field is given by, 
\begin{align} \label{eq:LXu_local}
\LXu(x,\alpha) = \left( x,\alpha, r_x(\mathsf{u}(x)), d_x\mathsf{u}(r_x(\alpha))+C_x(\alpha,\mathsf{u}(x)) \right).
\end{align}
\item The local flow $\mathfrak{F}^t_\mathfrak{u}$ of $\LXu$ is a vector-bundle isomorphism covering $\varphi_t:=\Fl^t_{\rho(\mathfrak{u})}$ the local flow of $\rho(\mathfrak{u})$. For every local dual section $\sigma$ and $x\in U$,
\begin{align}\label{eq:linfield_loc}
L_\mathfrak{u}\sigma = \left. \frac{d}{dt} \right|_{t=0} \left(\mathfrak{F}^t_{\mathfrak{u},x})^T(\sigma\circ\varphi_t) (x)\right). \end{align} 
 where $(\mathfrak{F}_{\mathfrak{u},x}^t)^T\colon A^*_{\varphi_t(x)}\to A^*_x$ is the transpose\footnote{For a continuous linear map $L\colon E\to F$, its transpose is $L^T\colon F^*\to E^*, L^T(\lambda):=\lambda\circ L$.} of $\mathfrak{F}_\mathfrak{u}^t \colon A_x\to A_{\varphi_t(x)}$.
 \item  Using the flows $\mathfrak{F}^t_\mathfrak{u}$ and $\varphi_t$ from 2., for every local section $\mathfrak{v}$, 
\begin{align}\label{eq:bracket_deriv_lin_field}
\left. \frac{d}{dt} \right|_{t=0} (\mathfrak{F}^t_{\mathfrak{u},x})^*\mathfrak{v}(x) = \LB[\mathfrak{u},\mathfrak{v}]_A(x), \quad x\in U
\end{align}
where $(\mathfrak{F}^t_\mathfrak{u})^*\mathfrak{v}(x) = (\mathfrak{F}^{t}_{\mathfrak{u},x})^{-1} \bigl( \mathfrak{v}(\varphi_t(x)) \bigr)$.
 \end{enumerate}
\end{theo}

\begin{proof}
In a bundle trivialisation define a local representative $\LXu^\alpha$ via formula \eqref{eq:LXu_local}. Note that $\LXu^\alpha$ is smooth in $(x,\alpha)$ and as the fourth component is linear in $\alpha$. As the local representative of $T\pi$ in a trivialisation is the projection $\text{pr}_{13}$ on the first and third component we obtain $\text{pr}_{13}\circ \LXu^\alpha (x,\alpha)=(x, r_x(\mathsf{u}(x)))$. This shows that $(\LXu^\alpha, r(\mathsf{u}))$ is a linear vector field and the second equality from \eqref{Xs}. 
Consider a local dual section $\sigma$. As $\Phi_\sigma (x,\alpha)= \langle \sigma (x),\alpha\rangle$ the computation after \eqref{Lsigma} yields 
\begin{align*} 
\LXu^\alpha (\Phi^\sigma)(x,\alpha) =&\langle d_x\sigma(r_x(\mathsf{u}(x))),\alpha\rangle + \langle\sigma(x), d_x\mathsf{u}(r_x(\alpha))+C_x(\alpha,\mathsf{u}(x))\rangle\\ =& \Phi^{L_u\sigma}(x,\alpha).
\end{align*} 
where the last equality follows by skew-symmetry of $C_x$. So for the local field $\LXu^\alpha$ both properties of \eqref{Xs} are satisfied. 

We now show that the two intrinsic properties from \eqref{Xs} uniquely identify the linear vector field. Let $X'$ be another vector field which satisfies the properties on the domain of the trivialisation. 
Fix $a\in A_x$ and since $\LXu^\alpha (a)$ and $X'(a)$ project to the same element in $T_xM$, their difference is vertical: $\LXu^\alpha(a)-X'(a) \in V_aA := \ker T_a\pi$. The canonical vertical-lift map $\mathrm{vl}_a\colon A_x\to V_aA$ is a linear isomorphism, \Cref{banachconnection}. 
Hence there is a unique $w_a\in A_x$ such that $\LXu^\alpha(a)-X'(a)=\mathrm{vl}_a(w_a)$. Let $\sigma$ be any local smooth section of $A^*$ defined in an $x$-neighborhood. 
Since $\Phi^\sigma(a+sw_a) = \langle\sigma(x),a+sw_a\rangle$, we have \[ d_a\Phi^\sigma(\mathrm{vl}_a(w_a)) = \left. \frac{d}{ds} \right|_{s=0} \Phi^\sigma(a+sw_a) = \langle\sigma(x),w_a\rangle. \] On the other hand, $\LXu^\alpha$ and $X'$ have the same action on $\Phi_\sigma$ and thus
\begin{align*} 0 = \LXu^\alpha (\Phi^\sigma)(a)-X'(\Phi^\sigma)(a) = d_a\Phi_\sigma\bigl(\LXu^\alpha (a)-X'(a)\bigr) = \langle\sigma(x),w_a\rangle. 
\end{align*} 
Every functional $\lambda\in A_x^*$ occurs as the value $\sigma(x)$ of some local smooth section of $A^*$ (in the local trivialisation around $x$ simply choose the constant map with value $\lambda$. It follows that $\lambda(w_a)=0$ for every $\lambda\in A_x^*$ and since the continuous dual of the Banach space $A_x$ separates points, we conclude that $w_a=0$ and the two sections coincide at $a$. We deduce that the two properties uniquely determines the vector field. 
As the properties from \eqref{Xs} are intrinsic, this establishes immediately, that the local vector fields $\LXu^\alpha$ for some trivialisations glue to a unique smooth vector field $\LXu$ on $A_U$. We obtain a unique linear vector field $(\LXu,\rho(\mathfrak{u}))$ which satisfies \eqref{Xs} thus establishing 1.
\smallskip

For both 2. and 3. we fix $x\in U$ and work locally in a bundle trivialisation around $x$. Let $\mathfrak{F}_\mathfrak{u}^t$ denote the local flow of $\LXu$, and $\varphi_t:=\Fl^t_{\rho{\mathfrak{u}}}$ denote the local flow of $\rho(\mathfrak{u})$. Restricting $t$ we may assume that the flow $\varphi_t(x)$ remains in the domain of the trivialisation for all suitable $t$. In this trivialisation we may thus define the operator $B_\mathsf{u}(y).\alpha :=d_y\mathsf{u}(r_y(\alpha))+C_y(\alpha,\mathsf{u}(y)) \in L(\mathbb{A},\mathbb{A})$ and note that the map $y \mapsto B_\mathsf{u}(y)$ is smooth in the operator-norm topology by construction.

2.  As $(\LXu,\rho(\mathfrak{u}))$ is a linear vector field, the assertion concerning the relation between the flows follows immediately from \Cref{chalinX}. Thus we obtain a continuous linear isomorphism $\mathfrak{F}_\mathfrak{u}^t \colon A_x \rightarrow A_{\varphi_t(x)}$. Write $Y_x(t)\in L(\mathbb{A}, \mathbb{A})$ for the the local representative of the family of linear isomorphisms. As $Y_x$ is the flow of the linear vector field $\LXu$, \eqref{eq:LXu_local} implies for every  $\alpha_0 \in \mathbb{A}$ that
\[  \frac{d}{dt} \bigl(Y_x(t).\alpha_0\bigr) = B_\mathsf{u}(\varphi_t(x)).(Y_x(t).\alpha_0) \] 
where the lower dot being application of a linear map. 
We deduce that 
\begin{align*}
Y_x(t).\alpha-Y_x(s).\alpha = \int_s^t B_\mathsf{u} (\varphi_r(x)).(Y_x(r).\alpha)dr.
\end{align*}
On every compact time interval and for every $\alpha_0$, the curve $r\mapsto Y_x(r).\alpha_0$ is continuous for each $\alpha\in\mathbb{A}$, and hence bounded. Applying the uniform boundedness principle \cite[23.14]{Sche97}. Let $J$ be a compact time interval on which the flow is defined, then continuity of $B_\mathsf{u}$ together with the integral equation yields 
\begin{align*}
\lVert Y_x(t)-Y_x(s)\rVert_\text{op} \leq |t-s|\sup_{r\in J}\lVert B_\mathsf{u}(\varphi_r(x))\rVert_{\text{op}}\sup_{r\in J}\lVert Y_x(r)\rVert_{\text{op}}
\end{align*}
Thus $Y_x$ is operator-norm continuous and satisfies the operator-valued integral equation $Y_x(t) = \operatorname{id}_{\mathbb{A}} + \int_0^t B_\mathsf{u}(\varphi_s(x)).Y_x(s)\,ds$ and the integrand is continuous with values in the Banach space $L(\mathbb{A},\mathbb{A})$. Consequently, by the fundamental theorem for Bochner integrals, $t\mapsto Y_x(t)$ is differentiable in operator norm and satisfies the operator-valued linear equation 
\begin{align}\label{eq:Yx_fibre}
\dot Y_x(t)=B_\mathsf{u}(\varphi_t(x))\circ Y_x(t), \ Y_x(0)=\operatorname{id}_{\mathbb{A}}.
\end{align}
For $\sigma \in \Gamma(A_U^\ast)$ define pointwise
$\sigma_t (x):=(\mathfrak{F}_{\mathfrak{u},x}^t)^T(\sigma(\varphi_t(x))) \in A_x^\ast, \quad x\in U$.  
As the transpose map sending $Y_x(t)$ to the transpose is continuous and linear, whence differentiable in operator norm, It follows that $\sigma_t$ is a differentiable $A^*_x$-curve. For $a\in A_x$ 
\begin{align*}
\langle \sigma_t (x),a\rangle = \langle (\sigma(\varphi_t(x))), \mathfrak{F}_\mathfrak{u}^t(a)\rangle =\Phi^\sigma (\mathfrak{F}_\mathfrak{u}^t(a))
\end{align*}
Using the definition of the flow and \eqref{Xs} we get
\begin{align*} \left. \frac{d}{dt} \right|_{t=0} \langle\sigma_t(x),a\rangle &= \left. \frac{d}{dt} \right|_{t=0} \Phi^\sigma(\mathfrak{F}^t_\mathfrak{u}(a))= \LXu (\Phi_\sigma)(a)= \Phi^{L_\mathfrak{u}\sigma}(a)= \langle(L_\mathfrak{u}\sigma)(x),a\rangle. \end{align*} 
Both sides are elements of $A_x^*$ which agree on every $a\in A_x$, and thus they are equal proving \eqref{eq:linfield_loc}. \smallskip

3. As in 2. we consider in the trivialisation $Y_x(t)\in L(\mathbb{A},\mathbb{A})$ the local representative of the fibre isomorphism $\mathfrak{F}^t_{\mathfrak{u},x}$. Then $(\mathfrak{F}^t_\mathfrak{u})^*\mathfrak{v}(x)$ is identified with $Y_x(t)^{-1}.\mathsf{v}(\varphi_t(x))$. We deduce from \eqref{eq:Yx_fibre} that  
$\left. \frac{d}{dt} \right|_{t=0} Y_x(t)^{-1} = -B_\mathsf{u}(x)$. Using these identities together with the product and chain rules  we obtain in the trivialisation for $\left. \frac{d}{dt} \right|_{t=0} (\mathfrak{F}^t_\mathfrak{u})^*\mathfrak{v}(x)$ the formula
\begin{align*}
 \left. \frac{d}{dt} \right|_{t=0} Y_x(t)^{-1}.\mathsf{v}(\varphi_t(x))&= -B_\mathfrak{u}(x)\mathsf{v}(x) + d_x\mathsf{v}(r_x(\mathsf{u}(x))) \\ &= d_x\mathsf{v}(r_x(\mathsf{u}(x))) - d_x\mathsf{u}(r_x(\mathsf{v}(x))) - C_x(\mathsf{v}(x),\mathsf{u}(x)). 
\end{align*}
Note that by skew-symmetry of $C_x$ the last identity shows that the derivative equals $\LB[\mathfrak{u},\mathfrak{v}]_A(x)$ by \eqref{loctrivrbracket}. Although the calculation is performed in a local trivialization, its final expression is $\LB[u,v]_A(x)$, and therefore independent of the chosen trivialization.
\end{proof}

\begin{defi}\label{completelift} 
Let $\mathfrak{u}\in \G(A_{ U})$ be a local section. The pair $(\LXu,\rho(\mathfrak{u}))$  from \Cref{thm:complete-lift} is called the \emph{complete lift} of $\mathfrak{u}$ on $A_{U}$. The flow $\mathfrak{F}_\mathfrak{u}^t$ of $\mathcal{X}_\mathfrak{u}$ is called the \emph{infinitesimal flow} of $\mathfrak{u}$.
\end{defi}

\begin{rem} \label{rem:admiss}
Let $\mathfrak{u}\in\Gamma(A_U)$, $\mathfrak{F}_\mathfrak{u}^t$ its infinitesimal flow, and $\varphi_t:=\Fl^t_{\rho(\mathfrak{u})}$. 
\begin{enumerate} 
\item For every Banach almost Lie algebroid, the projectability identity (property 2 from \eqref{Xs}) $T\pi\circ \LXu=\rho(\mathfrak{u})\circ\pi$ implies \[ \pi\circ \mathfrak{F}_\mathfrak{u}^t=\varphi_t\circ\pi \] wherever the flows are defined. 
\item If $A$ is a pre-Lie algebroid, then 
$\rho(\LB[\mathfrak{u},\mathfrak{v}]_A)=\LB[\rho(\mathfrak{u}),\rho(\mathfrak{v})]$ for all local sections $\mathfrak{v}$. Combining this identity with \eqref{eq:bracket_deriv_lin_field}, one obtains $\rho\bigl((\mathfrak{F}_\mathfrak{u}^t)^*\mathfrak{v}\bigr) = (\varphi_t)^*(\rho(\mathfrak{v}))$, since both sides satisfy the same evolution equation. Equivalently, the fibre maps satisfy $\rho\circ \mathfrak{F}_{\mathfrak{u},x}^t = T_x\varphi_t\circ\rho$ 
whenever defined. 
\item If $A$ is a Lie algebroid, then 
$\LB[\LXu,\mathcal{X}_\mathfrak{v}]=\mathcal{X}_{\LB[\mathfrak{u},\mathfrak{v}]_A}$. 
Indeed, both project to $\LB[\rho(\mathfrak{u}),\rho(\mathfrak{v})]=\rho(\LB[\mathfrak{u},\mathfrak{v}])$. Then on fibrewise linear functions the identity is equivalent to 
$\LB[L_\mathfrak{u},L_\mathfrak{v}]=L_{\LB[\mathfrak{u},\mathfrak{v}]_A}$, which follows from the Jacobi identity. Now \Cref{thm:complete-lift} 1. implies equality. Consequently, the infinitesimal flow acts by local Lie algebroid automorphisms. More precisely, \[ (\mathfrak{F}_\mathfrak{u}^t)^*\LB[\mathfrak{v},\mathfrak{w}]_A = \LB[(\mathfrak{F}_\mathfrak{u}^t)^*\mathfrak{v},(\mathfrak{F}_\mathfrak{u}^t)^*\mathfrak{w}]_A \] 
whenever all expressions are defined.
\end{enumerate} 
\end{rem}

\begin{setup}\label{setup:PF_PB}
Let $\mathfrak{u}\colon I\times U\to A$ be a smooth time-dependent local section, and let $\FlXu$ denote the evolution map of the time-dependent linear vector field $\mathcal{X}_{\mathfrak{u}_t}$. Let $\varphi_{t,s}:=\Fl^{t,s}_{\rho(\mathfrak{u})}$ 
be the evolution map of the time-dependent vector field $\rho(u_t)$. For a local section $\mathfrak{v} \in \Gamma(A_U)$, we define the \emph{pullback by the infinitesimal flow of $\mathfrak{u}$} 
\[ (\mathfrak{F}_\mathfrak{u}^{t,s})^*\mathfrak{v}(x) := (\mathfrak{F}_{\mathfrak{u},x}^{t,s})^{-1} \bigl(\mathfrak{v}(\varphi_{t,s}(x))\bigr) = \mathfrak{F}_{\mathfrak{u},\varphi_{t,s}(x)}^{s,t} \bigl(\mathfrak{v}(\varphi_{t,s}(x))\bigr), \]
where $\varphi_{t,s}$ is the flow of $\rho(\mathfrak{u})$. Define the \emph{push-forward by the infinitesimal flow of $\mathfrak{u}$} via
\[ (\mathfrak{F}_\mathfrak{u}^{t,s})_*\mathfrak{v}(y) := \mathfrak{F}_{\mathfrak{u},\varphi_{s,t}(y)}^{t,s}\bigl(\mathfrak{v}(\varphi_{s,t}(y))\bigr). \] 
These operators are inverse to one another on their common domains.
\end{setup}

\begin{lem} \label{timedebracketsection}
 For a local section $\mathfrak{v}$ of $A_U$ we use the notation from \Cref{setup:PF_PB}. 
Then, for every $x \in U$ for which the expressions are defined, 
\begin{align}\label{brackettimesection}
\left. \frac{\partial}{\partial t} \right|_{t=s} (\mathfrak{F}^{t,s}_\mathfrak{u})^*\mathfrak{v}(x) = \LB[\mathfrak{u}_s,\mathfrak{v}]_A(x). 
\end{align}
\end{lem}
\begin{proof} 
Fix $s\in I$ and $x\in U$. By the definition $\mathfrak{F}_{\mathfrak{u}}^{s,s}=\mathrm{id}_A$. Moreover, the derivative with respect to the first time variable at $t=s$ is the time-$s$ vector field $\left. \frac{\partial}{\partial t} \right|_{t=s} \FlXu (a) = \mathcal{X}_{\mathfrak{u}_s}(a)$ for every $a$ in the relevant domain.
Similarly,  $\left. \frac{\partial}{\partial t} \right|_{t=s} \varphi_{t,s}(x) = \rho(\mathfrak{u}_s)(x)$. Consequently, the derivative at \(t=s\) of $(\mathfrak{F}^{t,s}_\mathfrak{u})^*\mathfrak{v}(x)$ depends only on the linear vector field $\mathcal{X}_{\mathfrak{u}_s}$. Applying \eqref{eq:bracket_deriv_lin_field} we obtain \eqref{brackettimesection}.
\end{proof}

\begin{rem}\label{cktime} 
Denote by $\mathfrak{u},\mathfrak{v}\colon I \times U \rightarrow A$ two time-dependent  sections of $A$. If $\mathfrak{u},\mathfrak{v}$ are of class $C^{k,\infty}$, then the bracket $(t,x)\mapsto \LB[\mathfrak{u}_t,\mathfrak{v}_t]_A(x)$ is a $C^{k,\infty}$-section.

For fixed $x\in U$, the fibre maps $\mathfrak{F}^{t,s}_{\mathfrak{u},x}\colon A_x\to A_{\varphi_{t,s}(x)}$ are $C^k$ in $(t,s)$ and smooth in the spatial variables. To see this, by the same integral and uniform boundedness argument as in the proof of \Cref{thm:complete-lift}, in a local trivialisation the fibre map $\mathfrak{F}^{t,s}_{\mathfrak{u},x}$ is represented by the evolution operator $Y_x(t,s)\in L(\mathbb{A},\mathbb{A})$ of a non-autonomous linear equation whose evolution satisfies an operator valued integral equation in $L(\mathbb{A},\mathbb{A})$. The same is true for its inverse and thus the induced pullback operators $(\mathfrak{F}^{t,s}_{\mathfrak{u},x})^*\colon A_{\varphi_{t,s}(x)}\to A_x$ also depend smoothly on $(t,s)$ with respect to the operator norm (i.e. in a trivialisation they yield smooth maps into $L(\mathbb{A},\mathbb{A})$).
\end{rem} 

\begin{lem}\label{lem:pullback-differentiation} For a Banach pre-Lie algebroid $(A,M,\rho,\LB_A)$ and any smooth time-dependent local section $\mathfrak{u}\colon I\times U \rightarrow A$ and $Z\colon I\times U \rightarrow A$ be a $C^{1,\infty}$-local section. Fix $s \in I$ and define pointwise $Q_t:=(\mathfrak{F}_\mathfrak{u}^{t,s})^*Z_t$. Then for all $x\in U$ for which the expressions are defined,
\begin{align} \label{wq:flow:differential}
\frac{\partial Q_t}{\partial t}(x) = (\mathfrak{F}_\mathfrak{u}^{t,s})^* \Big( \frac{\partial Z_t}{\partial t} - \LB[Z_t,\mathfrak{u}_t]_A \Big)(x). \end{align} \end{lem}

\begin{proof}
 Fix $s\in I$ and pick $x\in U$. The curves $t\mapsto Q_t(x)$,  $t\mapsto Z_t(x)$ take values in the Banach space $A_x$. By \Cref{cktime}, $t\mapsto (\mathfrak{F}^{t,s}_{\mathfrak{u},x})^*$ is differentiable, whence
 \[ \frac{\partial Q_t}{\partial t}(x) = \left( \frac{d}{dt}(\mathfrak{F}^{t,s}_{\mathfrak{u},x})^*\right)(Z_t)(x) + (\mathfrak{F}^{t,s}_{\mathfrak{u},x})^*\left(\frac{\partial Z_t}{\partial t}(x)\right).\]
The cocycle property $(\mathfrak{F}_\mathfrak{u}^{t+h,s})^\ast=(\mathfrak{F}_\mathfrak{u}^{t,s})^\ast\circ (\mathfrak{F}_\mathfrak{u}^{t+h,t})^\ast$ together with  \Cref{timedebracketsection} yields for the derivative: $(\mathfrak{F}_\mathfrak{u}^{t+h,s})^*\mathfrak{v} = (\mathfrak{F}_\mathfrak{u}^{t,s})^\ast\mathfrak{v} +h(\mathfrak{F}_\mathfrak{u}^{t,s})^\ast\LB[\mathfrak{u}_t,\mathfrak{v}]_A + o(h)$ (we suppress the trivialisation needed to make sense of the operator valued identity). Therefore $\frac{d}{dt}(\mathfrak{F}_\mathfrak{u}^{t,s})^*(\mathfrak{v}) = (\mathfrak{F}_\mathfrak{u}^{t,s})^*\LB[\mathfrak{u}_t,\mathfrak{v}]_A$ and inserting this into the product rule, skew symmetry of the Lie bracket yields \eqref{wq:flow:differential}.
\end{proof}

There is a dual theory of (local) linear vector fields on the dual bundle $A^\ast$, cf. \cite{Mac05}. As this needs the notion of a partial almost Poisson structure we record it in \Cref{app:dualfield}.

\subsection{Homotopy of admissible curves}\label{Homotopy}
Throughout this section, we work with a Banach Lie algebroid $(A,\pi,M,\rho,\LB_A)$ which satisfies the assumption {\bf (H)} (cf. Assumption \ref{H})). We will repeatedly work in local coordinates using the notation from \Cref{loctriv}. The notion of vertical bundles and vertical lifts can be found in \Cref{banachconnection}. 
The following auxiliary construction and following Lemma will be used in the proof of \Cref{P_existH}.

\begin{defi}\label{def:homotopy}
Let $c\colon [t_0,t_1]\to A$ be a smooth admissible curve and set $\gamma:=\pi\circ c$. Let $b\in \Gamma(\gamma^*A)$ be a smooth section along $\gamma$. Then there is a well-defined section $\mathcal{T}_c(b)\in \Gamma(c^*TA)$ defined as follows: Let $A_U\cong U\times\mathbb{A}$ be a vector-bundle trivialization such that $\gamma(J)\subseteq U$ for a subinterval of $[t_0,t_1]$. Write $r_x\colon \mathbb{A}\to T_xM$ for the local representative of the anchor, \Cref{loctho}. In the trivialisation we identify
\begin{align} c(t)=(\gamma(t),a(t)),\quad b(t)=(\gamma(t),\beta(t)),  \text{ then define }  \notag\\ 
\mathcal{T}_c(b)(t) = \left( \gamma(t), a(t), r_{\gamma(t)}(\beta(t)), \frac{d\beta}{dt}(t)+C_{\gamma(t)}(a(t),\beta(t)) \right), \label{eq:transport_operator}
\end{align} where $C$ is the local structure map from \Cref{locbrac} \eqref{loctrivrbracket}. 
\end{defi}
\begin{lem} \label{lem:Tc-local-formula} 
 Let $c\colon [t_0,t_1]\to A$ be a smooth admissible curve and set $\gamma:=\pi\circ c$. The expression \eqref{eq:transport_operator} is independent of the local trivialisation. For fixed $c$ we obtain an $\R$-linear map 
 \begin{align}
 \mathcal{T}_c \colon \Gamma(\gamma^\ast A) \rightarrow \Gamma(c^\ast(TA)), b \mapsto \mathcal{T}_c(b).\end{align}
 Finally, if $f\colon [t_0,t_1]\to\R$ is smooth, then 
 \begin{align}\label{eq:transport_Leibniz} \mathcal{T}_c(fb) = f\mathcal{T}_c(b)+\frac{df}{dt} b^V, 
 \end{align} 
 where $b^V$ denotes the vertical lift of $b$. 
\end{lem}
\begin{proof}
Unravelling the local conventions, \eqref{eq:transport_operator} clearly defines an element in $T_{c(t)}A$ depending smoothly on $t$. We have to prove that this is independent of the chosen trivialization, let $ \Psi\colon A_U\cong U\times\mathbb{A},\ \Psi'\colon A|_{U'}\cong U'\times\mathbb{A}$ be two local trivializations with non-empty overlap containing $\gamma(t)$ and write $h\colon U\cap U'\to \mathrm{GL}(\mathbb{A})$ for the change of trivialisation such that $\Psi'\circ \Psi^{-1}(x,\alpha)=(x,h(x).\alpha)$ (lower dot denoting application of a linear map). Along $\gamma$ we have on $A$ and $TA$  
\begin{align}\label{coord_trafo1}
a'=h.a,\quad  \beta'=h.\beta, \quad (x,a,y,\eta) \mapsto (x,h.a,y,h.\eta+dh(x;y).a).
\end{align} 
Finally, we let $r_x$ and $r_x'$ be the local representatives of the anchor with respect to the two trivialisations. As $\rho$ is a vector bundle map, $r_x' (h(x).\alpha)=r_x(\alpha)$. We deduce that the first three components of \eqref{eq:transport_operator} are independent of trivialisation and it suffices to prove the following identity 
\begin{align}\label{target_identity_independence}
\frac{d\beta'}{dt} + C'_{\gamma}(a',\beta') = h(\gamma).\left(\frac{d\beta}{dt}+C_{\gamma}(a,\beta)\right) + dh(\gamma;r(\beta)).a.
\end{align} 
where again $C,C'$ are the local structure maps with respect to the trivialisations. Since $c$ is an admissible curve, the chain rule 
\begin{align*}
\frac{d\beta'}{dt}(t)= dh(\gamma(t);\frac{d\gamma}{dt}(t)).\beta+ h(\gamma(t)).\frac{d\beta}{dt} = dh(\gamma(t);r_{\gamma(t)}(a(t))).\beta+ h(\gamma(t)).\frac{d\beta}{dt}.
\end{align*}
Next compute the transformation rule for $C$, let $u,v$ local sections of $A$. The bracket is intrinsic, so we get for the local representatives from \Cref{locbrac}
\begin{align*} 
&h.(dv(r(u))-du(r(v))+C(u,v))=h.\LB[u,v]_A \\ 
=&\LB[u,v]_A'=(dv'(r'(u'))-du'(r'(v'))+C(u',v'))\\ =&d(h.v)(r'(h.u))-d(h.u)(r'(h.v))+C(h.u,h.v).
\end{align*} 
Now we derivate the evaluation of the linear map $h$ to obtain the product formula $d(h.u)(r(v))=dh(r(v)).u + h.du(r(v))$ and similarly for the r\^{o}les of $u,v$ interchanged. Plug this into the last identity and cancel common terms: 
\begin{align*}
C'(h.u,h.v) =h.C(u,v)-dh(r(u)).v +dh(r(v)).u
\end{align*}
Apply the formulae for $C'$ at $u=a$ and $v=\beta$ and the derivative of $\beta$ to the left hand of \eqref{target_identity_independence} (and suppress the evaluation at $\gamma(t)$) to obtain
\begin{align*} 
\frac{d\beta'}{dt} + C'(a',\beta') = dh(r(a)).\beta + h.\frac{d\beta}{dt} + h.C(a,\beta)- dh(r(a)).\beta + dh(r(\beta)).a. \end{align*}
which is the right hand side of \eqref{target_identity_independence} proving that the section is well-defined. 

Linearity of the map $\mathcal{T}_c$ follows directly from the local formula by linearity of the anchor and bilinearity of $C$. Finally, let $f\in C^\infty([t_0,t_1],\R)$, then $fb$ has local fibre coordinate $f\beta$. Hence 
\begin{align*}\mathcal{T}_c(fb) = \left(\gamma,a, f r(\beta), f\frac{d\beta}{dt}+\frac{df}{dt}\beta+fC_{\gamma}(a,\beta )\right) =  f+\mathcal{T}_c(b)+\frac{df}{dt} b^V_c,
\end{align*} where the last identity follows as the vertical lift of $b(t)\in A_{\gamma(t)}$ at $c(t)\in A_{\gamma(t)}$ is represented in coordinates as 
$b^V(t) = \left( \gamma(t), a(t), 0, \beta(t) \right)$.
\end{proof}

\begin{rem}
Note that the transport map $\mathcal{T}_c$ makes sense for any $C^k$-admissible curve $c$, $k\geq 1$. However, as it places a derivative on the section $\beta$, it is then a linear map from $C^k$-sections to $C^{k-1}$-sections. To avoid notational clutter, the above statement only deals with smooth sections and curves. We will however use the full statement for finitely often differentiable mappings. 
\end{rem}

\begin{lem}\label{vertX}
Let $\mathfrak{v}\in \Gamma(A_U)$ be a local section, and let $\LXu[v]$ be its complete lift. 
Then, if we denote by $\LB[\mathfrak{v},\mathfrak{u}]^V_A$ the vertical lift of  $\LB[\mathfrak{v},\mathfrak{u}]_A$, 
\begin{align} \label{rrelvertX}
\LXu[v](\mathfrak{u}(x)) = T_x\mathfrak{u}(\rho(\mathfrak{v}(x))) + (\LB[\mathfrak{u},\mathfrak{v}]_A(x))^V_{\mathfrak{u}(x)}. 
\end{align}
Let $c$ be admissible, $\gamma=\pi\circ c$, and let $b$ be a section along $\gamma$. If $\mathfrak{v}_t$ satisfies $\mathfrak{v}_t(\gamma(t))=b(t)$, then 
\begin{align}\label{eq:TCB_ext}
\mathcal{T}_c(b)(t) = \mathcal{X}_{\mathfrak{v}_t}(c(t)) + \left(\frac{\partial \mathfrak{v}_t}{\partial t}(\gamma(t))\right)^V_{c(t)}.
\end{align}
\end{lem}
\begin{proof}
Evaluating \eqref{eq:LXu_local} and $T_x\mathfrak{u}$ we get
\begin{align*}
\LXu[v](\mathsf{u}(x)) &= \Bigl( x,\mathsf{u}(x), r_x(\mathsf{v}(x)), d_x\mathsf{v}(r_x(\mathsf{u}(x)))+C_x(\mathsf{u}(x),\mathsf{v}(x)) \Bigr), \\ T_x\mathsf{u}(\rho(\mathsf{v}(x))) &= \Bigl( x,\mathsf{u}(x), r_x(\mathsf{v}(x)), d_x\mathsf{u}(r_x(\mathsf{v}(x))) \Bigr), 
\end{align*}
The difference of both terms is vertical and equals $(\LB[\mathsf{u},\mathsf{v}]_A)^V$ by \eqref{loctrivrbracket}.

Let $\mathfrak{v}_t$ be a time-dependent local extension of $b$, so that $\mathfrak{v}_t(\gamma(t))=b(t)$. In local coordinates, $\beta(t)=\mathsf{v}_t(\gamma(t))$ and therefore 
$\dot\beta(t) = \frac{\partial \mathsf{v}_t}{\partial t}(\gamma(t)) + d_{\gamma(t)}\mathsf{v}_t(r_{\gamma(t)}(a(t)))$. 
Using the local expression of the complete lift from \Cref{thm:complete-lift}, 
\[ \mathcal{X}_{\mathsf{v}_t}(c(t)) = \bigl( \gamma(t),a(t), r_{\gamma(t)}(\beta(t)), d_{\gamma(t)}\mathsf{v}_t(r_{\gamma(t)}(a(t))) + C_{\gamma(t)}(a(t),\beta(t)) \bigr). 
\] 
Comparing with \Cref{def:homotopy} yields \eqref{eq:TCB_ext}.
\end{proof}
According to the definition of homotopy for admissible curve in finite dimension (see \cite{CrFe11} and \cite{Mar02}), we introduce:

\begin{defi}\label{D_variationa}
A map $a \colon [0,1]\times[0,1]\to A$ such that $a_\varepsilon := a(\varepsilon,\cdot)$ is an admissible curve for every $\varepsilon \in [0,1]$ is called a \emph{variation of admissible curves} if $\g_\varepsilon:=\pi\circ a_\varepsilon \colon [0,1]\rightarrow M$ satisfies\begin{itemize}
\item $\g_\varepsilon(0)=x$ and $\g_\varepsilon(1)=y$, and both $x,y$ fixed for all $\varepsilon\in [0,1]$
\item
$\rho(a_\varepsilon)=\dis\frac{d\g_\varepsilon}{dt}$ for all $\varepsilon\in [0,1]$
\end{itemize}
\end{defi}

\begin{defi}\label{D_homotopy}
For $i=0,1$  let $c_i\colon [0,1] \to A$ be of class $C^k$, $k\geq 1$. We say that $c_0$ is \emph{homotopic} to $c_1$, if there exists $C^{k}$ maps $a,b\colon [0,1]\times[0,1]\rightarrow A$ with
 \begin{enumerate}
 \item  $a(0,\cdot )=c_0$ and ${a}(1,\cdot )=c_1$ and ${a}$ is a variation of admissible curves 
 \item $\gamma:=\pi\circ {a}=\pi\circ {b}$ and $\rho\circ {b}(\varepsilon,t)=\dis\frac{\partial \g}{\partial \varepsilon}(\varepsilon, t)$.
 \item $\dis\frac{\p {a}_\varepsilon}{\p \varepsilon}=\mathcal{T}_{a_\varepsilon }(b_\varepsilon)$
 \item For every $\varepsilon \in [0,1]$, $b_\varepsilon(0)=0=b_\varepsilon(1)$
 \end{enumerate}
 The pair $({a},{b})$ is called a ($C^k$-)\emph{homotopy} between $c_0$ and $c_1$.
 \end{defi}

\begin{rem}\label{variation} More generally one could have defined a variation of admissible curves $\g\colon I\times[0,1]\to A$ for some closed interval $I$. However, modulo a homothety on $\R$ we can always assume that $I=[0,1]$.
\end{rem}

Our aim is now to generalize results from \cite{CF03} on homotopies of admissible paths to the Banach case. For this, let us first outline how extension results from the previous sections can be applied to homotopies

\begin{setup}[Extending base paths and sections]
Consider a smooth map $a\colon [0,1] \to A$ and set $\pi\circ a=\g$. We write 
$a_\varepsilon(t)=a(\varepsilon,t)$ for a family of admissible curves with $a_0=a$ and $\gamma_\varepsilon :=\pi \circ a_\varepsilon$.
Consider an extension $\hat{\g}$ of $\g(\varepsilon,t):=\g_\varepsilon(t)$ to  $[0,1]\times]-\eta,1+\eta[$ for some $\eta>0$ given in 
Proposition \ref{timedepend}. If we assume that $M$ is smoothly paracompact \Cref{timedependent2} assertion (2), shows that for every $\varepsilon\in [0,1]$, there exists a  smooth    time-dependent   
section ${\mathfrak{u}}_\varepsilon(t,\cdot)$ 
defined on a neighbourhood $ \hat{V}$ of $\hat{\g}$,   whose projection on $\mathbb{R}$  contains an  interval $]-\eta,1+\eta[$, for some $\eta>0$, and such that $\mathfrak{u}_\varepsilon(t,{\g}_\varepsilon(t))=a_\varepsilon(t)$ for all $(t,\varepsilon) \in [0,1]\times ]-\eta,1+\eta[$.

Let $V$ be the projection on $M$ of $\hat{V}$. By compactness of the image of the curves, for the time-dependent vector fields  $\rho(\mathfrak{u}_\varepsilon)$  (resp. $\mathcal{X}_{\mathfrak{u}_\varepsilon}$),  
$t\in ]-\eta,1+\eta]$, for each $(t,x)\in \hat{V},\varepsilon\in [0,1]$, there exists $0<\eta'<\eta$ such that the flow  $\Fl_{\rho(\mathfrak{u}
_\varepsilon)}^{t',t}$  (resp. $\mathfrak{F}^{t',t}_{\mathfrak{u}_\varepsilon}$ ) is defined  on some  neigbourhood of $x$ (resp. $\mathfrak{u}_\varepsilon(t,x)$) such that $0\leq t\leq t'<1+\eta'$.
\end{setup}
Hence homotopies naturally fit into the machinery of extensions of sections. We continue with an auxiliary result concerning their evolution equations. The proof is postponed to \Cref{app:proof}.

\begin{lem} \label{lem:right-transport} 
Let $\mathfrak{u}\colon [0,1]\times U\to A$ be a $C^{k,\infty}$-time-dependent local section, $k\geq 1$, and assume that the evolution $\mathfrak{F}_\mathfrak{u}^{t,s}$ of the time-dependent complete lifts $\mathcal{X}_{\mathfrak{u}_t}$ is defined on the relevant domains for all $s,t\in[0,1]$.

Then the following assertions hold for the push-forward of sections. 
\begin{enumerate} 
\item For every local section $w$, the time-dependent section $Y_t:=(\mathfrak{F}_\mathfrak{u}^{t,s})_*w$ is for every $x$ for which the flows are defined the unique solution of 
\begin{align}\label{pushforward_transport}
\frac{\partial Y_t}{\partial t}(x) = \LB[Y_t,\mathfrak{u}_t]_A(x), \qquad Y_s(x)=w(x).
\end{align}
\item Let $f\colon [0,1]\times U\to A$ be a $C^{k,\infty}$-time-dependent local section. Define pointwise 
$v_t(y) := \int_0^t \mathfrak{F}^{t,r}_{\mathfrak{u},\varphi_{r,t}(y)} \left( f_r(\varphi_{r,t}(y)) \right)\,dr$ whenever all the terms are defined. For fixed $t,y$, the integral is the Bochner integral\footnote{cf. \cite[Sections 23-25]{Sche97} for the definition and basic properties used in the following.} in the Banach fibre $A_y$. Then $v$ is the unique $C^{k,\infty}$-solution of 
\begin{align}\label{eq:TD_pushforward}
\frac{\partial v_t}{\partial t} - f_t = \LB[v_t,\mathfrak{u}_t]_A, \qquad v_0=0.
\end{align}
\end{enumerate} \end{lem}

The central result in this section is the next theorem which mirrors \cite[Proposition 1.3]{CF03}. However, the proof of loc.cit.\, uses tools, e.g.\, compactly supported vector fields which are not available on infinite-dimensional Banach manifolds. Thus these arguments do not admit a direct generalization to Banach manifolds.

\begin{theo}\label{P_existH}
Assume that $M$ is smoothly paracompact and $k\geq 2$. Let $a\colon [0,1]^2\rightarrow A$ be a $C^k$-variation of admissible curves and $b\colon [0,1]^2\to A$ be a $C^k$-map such that $\gamma= \pi\circ a=\pi\circ b$, $b_\varepsilon(0)=b_\varepsilon(1)=0$, for every $\varepsilon$ and $b^t:=b(\cdot,t)$ is an admissible curve for all $t \in [0,1]$. We have the following equivalences
\begin{enumerate}
\item[\rm (1)] $(a,b)$ is a homotopy from $a_0$ to $a_1$;
\item[\rm (2)] the pair gives rise to an algebroid morphism  
\[f_{a,b}\colon T([0,1]\times[0,1])\to A, \quad f_{a,b}(\lambda \partial_t+\mu\partial_\varepsilon)=\lambda a+\mu b\]
over $\gamma \colon [0,1]^2 \rightarrow M, \gamma :=\pi \circ a$.
\item [\rm (3)] 
For one (hence every) extension $\mathfrak{u}_\varepsilon(t,\gamma_\varepsilon(t))=a_\varepsilon(t)$ from \Cref{timedependent2}, the section $b$ is given by:
\begin{equation}\label{bFu}
\dis b_\varepsilon(t)= \int_0^t \mathfrak{F}_{\mathfrak{u}_\varepsilon, \gamma_\varepsilon (s)}^{t,s} \frac{\partial \mathfrak{u}_\varepsilon}{\partial \varepsilon}\left(s,\g_\varepsilon(s)\right)ds.
\end{equation}
where the integral is the Bochner integral in the Banach space $A_{\gamma_\varepsilon(t)}$.
\end{enumerate}
\end{theo}

\begin{rem}\label{R_restriction01}
 For any closed interval $I=[t_0,t_1]$ in $\R$, the   \Cref{P_existH} is also true by replacing $[0,1]\times [0,1]$ by $[0,1]\times I$ after an  adequate adaptation of the notion of homotopy relative to the interval $I$.
 \end{rem}

\begin{rem}
 If $A$ is integrated by a Banach Lie groupoid, then the homotopy condition can also be characterized in terms of the Maurer--Cartan form of the integrating groupoid, cf. \cite{CF03}. This characterization carries over to the Banach setting and will be discussed in the sequel devoted to local integration.
\end{rem}
\begin{proof}[Proof of \Cref{P_existH}] $(1) \Leftrightarrow (2)$ We check \Cref{D_LieMorphism}. Now $a$ is a variation of admissible curves, \Cref{D_variationa}, and the two parameter family $\gamma = \pi \circ a = \pi \circ b$, property \ref{LM1} is satisfied as $a_\varepsilon$ and  
$b^t$ are admissible curves over $\g_\varepsilon$ and $\g^t$, respectively, for all $(\varepsilon,t)\in[0,1]^2$.

It suffices to check \ref{LM2} on the coordinate vector fields $\partial_t,\partial_\varepsilon$, since they generate $T([0,1]^2)$. Using formula \eqref{eq_LieMoprhism} and the local expression of the bracket from \Cref{locbrac}, we obtain from $\LB[\partial_t,\partial_\varepsilon]=0$ that
\begin{align}\label{eq:lie_hom_new}
0=f_{a,b} \Bigl( \LB[\partial_t,\partial_\varepsilon] \Bigr) = \frac{\partial a_\varepsilon}{\partial\varepsilon} -\frac{\partial b_\varepsilon}{\partial t}  + C_{\gamma(\varepsilon,t)} (b_\varepsilon,a_\varepsilon) = \frac{\partial a_\varepsilon}{\partial\varepsilon} - \mathcal{T}_{a_\varepsilon}(b_\varepsilon). \end{align}
where the second identity follows from \eqref{eq:transport_operator} as $C$ is skew-symmetric. 
Therefore condition \ref{LM2} is equivalent to the right hand side of \eqref{eq:lie_hom_new} vanishing which is precisely Condition~3. of \Cref{D_homotopy}. Further, we have assumed that $b_\varepsilon(0)=b_\varepsilon(1)=0$, i.e. Condition~(4) of \Cref{D_homotopy} holds. Hence $(1)\Leftrightarrow (2)$.
\smallskip

$(1) \Leftrightarrow (3):$ Let $\gamma(\varepsilon,t) = \pi(a(\varepsilon,t))$ and choose by \Cref{timedependent2} 2., a $C^{k,\infty}$-family of time-dependent local sections $ \mathfrak{u}$ defined on a neighbourhood of the image of $\gamma$ such that $\mathfrak{u}_\varepsilon(t,\gamma_\varepsilon(t)) = a_\varepsilon(t)$ for every $\varepsilon\in[0,1]$. For every fixed $\varepsilon$ we denote by $\mathfrak{F}^{t,s}_{\mathfrak{u}_\varepsilon}$ the evolution of $\mathfrak{u}_\varepsilon$ and by $\varphi^{t,s}_\varepsilon:=\Fl_{\rho(\mathfrak{u}_\varepsilon)}^{t,s}$ the evolution of $\rho(\mathfrak{u}_\varepsilon)$.
Set $f_\varepsilon = \frac{\partial \mathfrak{u}_\varepsilon} {\partial\varepsilon}$. 
By \Cref{lem:right-transport} 2., the section 
\[ v_\varepsilon(t,y) = \int_0^t \mathfrak{F}_{\mathfrak{u}_\varepsilon,\varphi_\varepsilon^{s,t}(y)}^{t, s} (f_\varepsilon(s))(\varphi_\varepsilon^{s,t}(y))\,ds,\qquad x\in U \] is the unique solution of 
\begin{align}\label{eq:vt_partial_diffeq}
\frac{\partial}{\partial t} v_\varepsilon (t)(x) - \frac{\partial}{\partial \varepsilon}  \mathfrak{u}_\varepsilon (t) (x)= \LB[v_\varepsilon(t), \mathfrak{u}_\varepsilon(t)]_A(x), \qquad v_\varepsilon(0,\cdot)=0, \quad x\in U. 
\end{align} 
Set $X_\varepsilon:=\rho(\mathfrak{u}_\varepsilon)$, $Y_\varepsilon:=\rho(v_\varepsilon)$. Applying the anchor $\rho$ to \eqref{eq:vt_partial_diffeq} and using that $\rho$ is a Lie algebra morphism (\Cref{C_rhoLieAlgebraMorphism}, as $A$ is a Lie algebroid) yields 
\begin{align}\label{anchored_PDE}
\frac{\partial Y_\varepsilon}{\partial t} (t,x) - \frac{\partial X_\varepsilon}{\partial \varepsilon}(t,x) = \LB[Y_\varepsilon(t),X_\varepsilon(t)](x), \quad x\in U.  \end{align}
 Now since $a_\varepsilon$ is an admissible curve extended via $\mathfrak{u}_\varepsilon$ we get 
\[ X_\varepsilon ( t,\gamma_\varepsilon(t) ) = \rho(\mathfrak{u}_\varepsilon(t,\gamma_\varepsilon(t))=\rho(a_\varepsilon(t))= \frac{\partial\gamma_\varepsilon} {\partial t}(t), \] 
We differentiate with respect to $\varepsilon$. The mixed partial derivatives commute due to the map $\gamma$ being $C^2$. Denoting by $d_2$ differentiation in the $M$-variable, the chain rule yields:
\begin{align}\label{eq:Xeps}
\frac{\partial}{\partial t} \frac{\partial}{\partial \varepsilon}\gamma_\varepsilon (t) =\frac{\partial}{\partial \varepsilon} X_\varepsilon (t,\gamma_\varepsilon(t))=  \frac{\partial X_\varepsilon}{\partial \varepsilon} (t,\gamma_\varepsilon(t))+d_2X_\varepsilon \left(t,\gamma_\varepsilon(t);\frac{\partial}{\partial \varepsilon } \gamma_\varepsilon(t)\right)
\end{align}
The chain rule and the derivative formula for $\gamma_\varepsilon$ yield 
\begin{align}
\frac{\partial}{\partial t} Y_\varepsilon (t,\gamma_\varepsilon (t)) \stackrel{\hphantom{\eqref{anchored_PDE}}}{=}& \frac{\partial Y_\varepsilon}{\partial t} (t,\gamma_\varepsilon (t)) + d_2 Y_\varepsilon\left(t,\gamma_\varepsilon (t);\frac{\partial}{\partial t} \gamma_\varepsilon(t)\right) \notag \\ \stackrel{\eqref{anchored_PDE}}{=}&\left(\frac{\partial X_\varepsilon}{\partial \varepsilon } + \LB[Y_\varepsilon,X_\varepsilon] + d_2 Y_\varepsilon \circ (\text{id}, X_\varepsilon)\right)(t,\gamma_\varepsilon(t)) \notag \\ \stackrel{\hphantom{\eqref{anchored_PDE}}}{=}&  \frac{\partial X_\varepsilon}{\partial \varepsilon} (t,\gamma_\varepsilon (t)) + d_2 X_\varepsilon (t,\gamma_\varepsilon(t); Y_\varepsilon (t,\gamma_\varepsilon(t))) \label{eq:Yeps}
\end{align}
Comparing \eqref{eq:Yeps} and \eqref{eq:Xeps}, $Y_\varepsilon(t,\gamma_\varepsilon (t))$ solves the same ODE as $\frac{\partial}{\partial \varepsilon} \gamma_\varepsilon(t)$.
Both solutions vanish at $t=0$ by \eqref{eq:vt_partial_diffeq} and as the endpoints are fixed by the homotopy and thus $\frac{\partial}{\partial \varepsilon }\gamma_\varepsilon (0)=0$. Hence uniqueness of linear ODEs in Banach spaces yields $Y_\varepsilon ( t,\gamma_\varepsilon(t) ) = \frac{\partial\gamma_\varepsilon} {\partial\varepsilon}(t)$. Therefore we can exploit now that $\gamma_\varepsilon$ is an integral curve of $\rho(\mathfrak{u}_\varepsilon)$ and by setting $ \widetilde b_\varepsilon(t) := v_\varepsilon(t,\gamma_\varepsilon (t))$ we obtain that 
\begin{align}\label{eq:btilde_def}
 \widetilde b_\varepsilon(t) = \int_0^t \mathfrak{F}_{\mathfrak{u}_\varepsilon,\gamma_\varepsilon(s)}^{t,s} \frac{\partial \mathfrak{u}_\varepsilon}{\partial \varepsilon}\left(s,\g_\varepsilon(s)\right)ds \quad \text{ and } \rho(\widetilde b_\varepsilon (t)) = \frac{\partial\gamma_\varepsilon} {\partial\varepsilon}(t). \\ \Rightarrow 
 \frac{\partial}{\partial t} \tilde{b}_\varepsilon (t)= \frac{\partial v_\varepsilon}{\partial t} (t,\gamma_\varepsilon(t))+d_2v_\varepsilon \left(t,\gamma_\varepsilon(t);\underbrace{\frac{\partial}{\partial t}\gamma_\varepsilon (t)}_{=\rho(\mathfrak{u}_\varepsilon (t))}\right)
 \label{deriv:btilde_chain}
\end{align}
Now differentiate $a_\varepsilon (t) =\mathfrak{u}_\varepsilon (t,\gamma_\varepsilon(t))$ in $\varepsilon$ and use the chain rule: 
\begin{align}\label{eq:2ndbtilde}
\frac{\partial a_\varepsilon} {\partial\varepsilon}(t) & \stackrel{\hphantom{+}\eqref{eq:btilde_def}\hphantom{\eqref{loctrivrbracket}}}{=} \left(\frac{\partial \mathfrak{u}_\varepsilon}{\partial \varepsilon}+ d_2 \mathfrak{u}_\varepsilon (\cdot; \rho(\tilde{b}_\varepsilon (t)))\right)  (t,\gamma_\varepsilon (t))\notag \\ &\stackrel{\hphantom{+}\eqref{eq:vt_partial_diffeq}\hphantom{\eqref{loctrivrbracket}}}{=} \left(\frac{\partial v_\varepsilon }{\partial t}- \LB[v_\varepsilon, \mathfrak{u}_\varepsilon]+d_2 \mathfrak{u}_\varepsilon (\cdot; \rho(\tilde{b}_\varepsilon (t)))\right) (t,\gamma_\varepsilon (t)) \notag\\ &\stackrel{\eqref{deriv:btilde_chain}+\eqref{loctrivrbracket}}{=} \frac{\partial\widetilde b_\varepsilon} {\partial t} - C_{\gamma_\varepsilon (t)} (\widetilde b_\varepsilon(t),a_\varepsilon(t)) \stackrel{\eqref{eq:transport_operator}}{=} \mathcal{T}_{a_\varepsilon(t)} (\widetilde b_\varepsilon (t)).
\end{align} 
Note that \eqref{eq:btilde_def} and \eqref{eq:2ndbtilde} are exactly properties 2. and 3. from \Cref{D_homotopy}. Thus if we assume now that $b=\tilde{b}$, i.e. 3. holds, we have $\tilde{b}_\varepsilon(0)=b_\varepsilon(0)=0, \tilde{b}_\varepsilon(1)=b_\varepsilon(1)=0$ as these hold for $b$. We conclude that $(a,\tilde{b})$ is a homotopy and $(3)\Rightarrow (1)$.

For the converse, assume that $(a,b)$ is a homotopy from $a_0$ to $a_1$. Thus \Cref{D_homotopy} 3. shows that $\mathcal{T}_{a_\varepsilon} (b)=\frac{\partial a_\varepsilon} {\partial\varepsilon}$. Now the curve $c_\varepsilon (t) := b_\varepsilon (t) - \tilde{b}_\varepsilon (t)$ covers $\gamma_\varepsilon$. Hence the local formula (cf. \Cref{locbrac}) yields \begin{align}\label{eq:ODE_lin}
\mathcal{T}_{a_\varepsilon}(c_\varepsilon (t))= \frac{\partial c_\varepsilon(t)}{\partial t} +  C_{\gamma_\varepsilon(t)}(a_\varepsilon(t), c_\varepsilon (t) ) = \mathcal{T}_{a_\varepsilon}(b_\varepsilon (t))-\mathcal{T}_{a_\varepsilon}(\tilde{b}_\varepsilon (t)) =0.
\end{align}
and $b_\varepsilon (0)=0$ as $(a,b)$ is a homotopy and $\tilde{b}_\varepsilon (0)=0$ by \eqref{eq:btilde_def}. Thus $c_\varepsilon(0)=b_\varepsilon(0)-\tilde{b}_{\varepsilon} (0)=0$. As \eqref{eq:ODE_lin} is a linear homogeneous ODE in the pullback bundle $\gamma_\varepsilon^\ast A$, we deduce from uniqueness of solutions to ODE on Banach manifolds, \cite{Lang01}, that $b=\tilde{b}$ and thus the formula \eqref{bFu} holds. Note that the assumption on $b$ enforces $\widetilde b_\varepsilon(1)=0$ as $b=\widetilde b_\varepsilon$ by uniqueness, thus finishing this part of the proof. 
 \end{proof} 

 \appendix
\section{Topological properties of infinite-dimensional manifolds}\label{app:smooth_para}

In this appendix we collect a few results on (smoothly) paracompact spaces needed in some of the proofs. In general we try to avoid the use of paracompactness and smooth bump functions since it is for many of our discussion an unnecessary restriction (in particular the requirement to be smoothly regular requires the existence of a suitably smooth norm on the Banach spaces). We need a standard topological result.

\begin{lem}\label{lem:paracpt-nbhd_cmpt.set}
Let $M$ be a Banach manifold which is regular as a topological space\footnote{For infinite-dimensional manifolds this is a relevant restriction, cf. the counterexample in \cite[27.6]{KrMi97} exhibiting a non-regular Hilbert manifold.} and $K \subseteq M$ be a compact set. Then there exists an open $K$-neighborhood in $M$ which is metrisable and thus paracompact. 
\end{lem}

\begin{proof}
For every $x\in K$ we pick a chart $(U_x,\kappa_x)$ with $x \in U_x$ such that $\kappa_x (x)=0$ and $\kappa_x(U_x)$ contains the closed unit ball $\overline{B_1(0)}$ in the Banach space model space. As $M$ is regular as a topological space, there exists an open $x$-neighborhood $N_x$ such that $\overline{N}_x^M \subseteq U_x$, where the superscript denotes the closure in $M$.  By compactness, there are finitely many $x_i \in K$ such that $K \subseteq O:= \bigcup_i \kappa_{x_i}^{-1} (B_1(0))\cap N_x$. Set $T:= \bigcup_i \kappa_{x_i}^{-1}(\overline{B_1(0)})\cap \overline{N_x}$. Then the closed balls $\overline{B_1 (0)}$ are metrisable, whence the $\kappa_{x_i}^{-1}(\overline{B_1(0)})\cap\overline{N}_x$ form a cover of $T$  by closed metrisable subsets. We deduce that $T$ is metrisable by \cite[Theorem 4.4.19]{Eng89}. Hence the open neighborhood $O \subseteq T$ is metrisable. 
\end{proof}

Recall from \cite[Section 15]{KrMi97} the following definitions:

\begin{defi}
Consider a manifold $M$ and a subset $\mathcal{S} \subseteq C^0(M,\R)$ of continuous functions. We say that $M$ 
\begin{enumerate}
\item admits $\mathcal{S}$-bump functions or is ($\mathcal{S}$ -regular) if for every $x\in M$ and every open $x-$neighborhood $U \subseteq M$ there exists $\chi \in \mathcal{S}$ such that $\chi (M) \subseteq [0,1], U = \chi^{-1} (]0, 1])$ and $\chi (x) = 1$.
\item is $\mathcal{S}$-paracompact if every open cover $(U_i)_{i \in I}$ of $M$ admits a subordinate partition
of unity $(\chi_i)_{i\in I}$ of functions $\chi_i \in \mathcal{S}, i \in I$.
\end{enumerate}
If $\mathcal{S} = C^k(M, \R)$ for some $k\in \N_0 \cup \{\infty\}$, we shorten the notation and say that $M$ admits $C^k$-bump functions or is $C^k$
-paracompact. Note that $C^0$-paracompactness is equivalent to the usual topological notion of paracompactness.
\end{defi}
Several remarks are in order: 

\begin{rem}
\begin{enumerate}
\item $\mathcal{S}$-paracompactness implies $\mathcal{S}$-regularity. 
\item By \cite[Theorem 16.10]{KrMi97} on a manifold with the Lindel\"{o}f property, being $\mathcal{S}$-regular implies $\mathcal{S}$-paracompactness.
\item Assume that $\mathcal{S}$ is a property which is stable under diffeomorphisms (e.g. $\mathcal{S}=C^k, k \in \N_0 \cup \{\infty\}$. If the model space of the manifold is $\mathcal{S}$-regular, also $M$ is $\mathcal{S}$-regular, whence $\mathcal{S}$-paracompact if $M$ is Lindel{\"o}f by 2.  
\end{enumerate}\end{rem}

\begin{lem}\label{lem:reg_BB}
For a Banach bundle $\pi \colon E \rightarrow M$, $M$ is ($C^\infty$-)paracompact if and only if $E$ is ($C^\infty$)-paracompact.
\end{lem}
\begin{proof}
If $M$ is ($C^\infty$)-paracompact, the assertion for $E$ is (a variant of) \cite[Proposition 29.7]{KrMi97}. Conversely, if $E$ is ($C^\infty-$)paracompact, we observe that the zero section $0_M$ identifies $M$ with a closed submanifold of $E$, so $M$ is paracompact by \cite[Corollary 5.1.9]{Eng89}. So any open cover $\mathcal{O}$ of $M$ admits a locally finite open refinement $\mathcal{R}$. As $\pi^{-1}(\mathcal{R}):=\{\pi^{-1}(R) \colon R \in \mathcal{R}\}$ is a locally finite open cover of $E$, if $E$ is $C^\infty$-paracompact we can find a $C^\infty$-partition of unity $\chi_R, R\in \mathcal{R}$ subordinate to the open cover of $E$, see \cite[Proposition 16.2]{KrMi97}. Then the family $\eta_R := \chi_R \circ 0_M, R\in \mathcal{R}$ is a smooth partition of unity subordinate to $\mathcal{R}$ and $M$ is $C^\infty$-paracompact.
\end{proof}
\begin{rem}
For Banach spaces the existence of $C^k$-bump functions, $C^k$-regularity and $C^k$-paracompactness is closely tied to the existence of a $C^k$-norm. So this condition is somewhat restrictive concerning the choice of Banach spaces. See \cite[Section 16]{KrMi97} for known criteria and (counter-)examples.
\end{rem}
\newpage
\section{Proofs for \Cref{timedependenhomotopy}}\label{app:proof}
In this appendix we collect two involved proofs for arguments needed in the treatment of (time-dependent) section and the evolution equations.

\subsection*{Proof of \Cref{timedepend}}

Our aim is to extend a $C^k$-map $\gamma \colon [0,1] \times [0,1] \rightarrow M$ in one variable such that it is constant near the endpoints. This can be achieved essentially by an application of the Seeley-extension theorem.
Recall that we have assumed that 
$$\g(\varepsilon,0)=\g(0,0) \text{ and }\g(\varepsilon,1)=\g(0,1), \text{ for all }\varepsilon\in [0,1]$$
We pick two charts $(U_0,\phi_0)$ and $(U_1,\phi_1)$ of $M$ with conneted domains such that $\g(\varepsilon,0)=\g(0,0)\in U_0$ and $\g(\varepsilon,1)=\gamma(0,1)\in U_1$ for all $\varepsilon\in[0,1]$

Since, for all $\varepsilon\in [0,1]$, we have  $\g_\varepsilon(0)=\g_0(0)$, we choose $\t\in ]0,1]$ such that $\g(t,\varepsilon)$ belongs to $U_0$ for all $ (\varepsilon,t)\in [0,1] \times [0,\t]$. Define $\alpha \colon [0,1]\times [0,\t] \rightarrow \phi_0(U_0)\subseteq \mathbb{M}, \alpha:=\phi_0 \circ \gamma$ to obtain a $C^k$-function with values in the Banach space $\mathbb{M}$. We apply a suitable version of the Seeley-extension theorem. For this combine \cite[Application 1 and Remark 4]{Han23}, for every $k\in \N_0\cup \{\infty\}$ there is a continuous linear extension operator 
$$\text{Ext}^k \colon C^k([0,1]\times [0,\tau),\mathbb{M}) \rightarrow C^k([0,1]\times (-\infty, \tau),\mathbb{M}).$$
We obtain a $C^k$-extension $\tilde{\alpha} := \mathrm{Ext}^k (\alpha|_{[0,1]\times [0,\tau)})$ (the restriction here is purely cosmetic to directly apply the Seeley-result)
We exploit that $\alpha([0,1]\times [0,\tau])\subseteq \phi_0(U_0)$ is compact, so by continuity of $\tilde{\alpha}$ there is $\eta >0$ small enough such that $\tilde{\alpha}([0,1]\times [-\eta,\tau))\subseteq \phi_0(U_0)$.  
Pick a smooth $h\colon \R\to [0,1]$ such that $h(t)=0$ for $t\leq -\frac{\eta}{2}$ and $h(t)=1$ for $t\geq -\frac{\eta}{4}$. Then we set 
$$\hat{\gamma}_0 \colon [0,1]\times [-\eta,\tau) \rightarrow U_1, \quad \hat{\gamma}_1(\varepsilon,t)=\phi_0^{-1}\circ \tilde{\alpha}(\varepsilon, h(t)t).$$
It follows that $\hat{\gamma}_0$ is a $C^k$-map. which  
extends $\gamma|_{[0,1]\times [0,\tau)}$ to the set $[0,1]\times [-\eta,\t)$ and such that 
$$\hat{\gamma}_0 (\varepsilon, t)=\phi_0^{-1}\circ \tilde{\alpha}(\varepsilon,0)=\phi^{-1}_0\circ \alpha(\varepsilon,0)=\gamma(\varepsilon,0), \quad \forall (\varepsilon, t) \in [0,1]\times \left(-\eta,-\frac{\eta}{2}\right].$$
 In the same way, there exists $\t < \t'<1$ such that $\g ([0,1]\times [\t',1])\subseteq U_1$. We apply the same reasoning to $\beta :=\phi_1\circ (\g|_{[0,1]\times [\t',1]})$. Shrinking  $\eta>0$ if necessary, we may assume that there is also $C^k$-extension $\hat{\gamma}_1 \colon [0,1] \times (\t',1+\eta] \rightarrow U_1$ of $\gamma|_{[0,1]\times (\t',1]}$ and such that 
$\hat{\gamma}_1 (\varepsilon,t)=\gamma(\varepsilon ,1)$ for all $(\varepsilon,t)\in [0,1]\times [1+\frac{\eta}{2},1+\eta]$. Now we consider the $C^k$-map defined by:
\[\hat{\g} \colon [0,1] \times [-\eta,1+\eta] \rightarrow M,\quad \hat{\g}(\varepsilon,t):= \hat{\g}_\varepsilon(t):=\begin{cases} 
\hat{\gamma}_1 (\varepsilon,t) & t \in [-\eta,\t) \\ \g (\varepsilon,t) & t \in [\t,\t']\\ 
\hat{\gamma}_p(\varepsilon, t) & t \in (\t', 1+\eta]\end{cases}\]
which is the required extension.

\subsection*{Proof of \Cref{lem:right-transport}}
Note first that by \Cref{setup:flowsCrs} the evolutions $\mathfrak{F}^{t,r}_{\mathfrak{u}},\varphi_{t,r}$ of the $C^{k,\infty}$-families of sections $\mathfrak{u},\rho(\mathfrak{u})$ are $C^{k,\infty}$-maps (with respect to the splitting $([0,1]\times [0,1])\times U$) for some $k\geq 1$. 

1. Let $Y\colon I\times U\to A$ be a $C^{k,\infty}$-section and fix $s\in I$. Fix $x\in U$ and choose $t$ so small that the evolutions are defined. Since $Q_t(x):=(\mathfrak{F}_\mathfrak{u}^{t,s})^*Y_t(x)\in A_x$, \Cref{lem:pullback-differentiation} applies pointwise in the fixed Banach fibre $A_x$ and yields
\[ \partial_tQ_t(x) = (\mathfrak{F}_\mathfrak{u}^{t,s})^* \Bigl( \partial_tY_t-[Y_t,\mathfrak{u}_t]_A \Bigr)(x). \]
Suppose first that $Y_t$ satisfies 
$\partial_tY_t=\LB[Y_t,\mathfrak{u}_t]_A,  Y_s=w$. Then $\partial_tQ_t(x)=0$ and therefore $Q_t(x)=Q_s(x)=w(x)$ for all $t$. As this is true for all $x$ in the flow domain, we can apply the inverse of $(\mathfrak{F}_{\mathfrak{u}}^{t,s})^*$, i.e. $(\mathfrak{F}_{\mathfrak{u}}^{t,s})_*$ on their common domain to obtain $Y_t=(\mathfrak{F}_\mathfrak{u}^{t,s})_*w$. Conversely, define $Y_t:=(\mathfrak{F}_\mathfrak{u}^{t,s})_*w$, then this formula yields a time-dependent $C^{k,\infty}$-section $Y_t$ (regularity follows from the chain rules discussed above) using injectivity of $(\mathfrak{F}_\mathfrak{u}^{t,s})^*$ one immediately establishes that it is a solution to \eqref{wq:flow:differential}.

For uniqueness, if $Z_t$ is another solution of \[ \partial_tZ_t(x)=\LB[Z_t,\mathfrak{u}_t]_A(x), \qquad Z_s(x)=w(x), \] then \Cref{lem:pullback-differentiation} implies that $ (\mathfrak{F}_\mathfrak{u}^{t,s})^*Z_t$ is constant. Evaluating at $t=s$ gives \[(\mathfrak{F}_\mathfrak{u}^{t,s})^*Z_t=w, \Rightarrow Z_t=(\mathfrak{F}_\mathfrak{u}^{t,s})^*w. \] Therefore the solution is unique.

2. We establish the statement first under the additional assumption that the bundle $A_U$ admits a global bundle trivialization $A_U\cong U\times\mathbb{A}$ such that all relevant curves remain in $U$. Consider the map $K(t,r,y) := \mathfrak{F}^{t,r}_{\mathfrak{u},\varphi_{r,t}(y)} \left( f_r(\varphi_{r,t}(y)) \right)$. For fixed $t,y$, $K(t,r,y)\in A_y$ for every $r\in[0,t]$, since $\varphi_{t,r}(\varphi_{r,t}(y))=y$.
Now $f$ is $C^{k,\infty}$ and the same holds for the flows. As we only compose in the spatial variable (where the maps are smooth), the flow the chain rules from \Cref{Crs_chainrules} imply that $K \colon ([0,1]\times [0,1])\times U \supseteq D \rightarrow \mathbb{A}$ is $C^{k,\infty}$. Standard arguments for differentiating parameter-dependent Bochner integration, \cite[Theorem 25.21]{Sche97}, shows that the mixed differentiability class is preserved, i.e. $v_t(y) = \int_0^tK(t,r,y)\,dr$ defines a $C^{k,\infty}$-time-dependent local section on every common flow domain.

Fix $x\in U$, and put $\gamma_x(t):=\varphi_{t,0}(x)$. Then $\frac{d}{dt}\gamma_x(t) = r_{\gamma_x(t)} (\mathsf{u}_t(\gamma_x(t)))$ and by the cocycle property $\varphi_{r,t}(\gamma_x(t)) = \gamma_x(r)$ we can rewrite $v_t$ as 
\begin{align}\label{eq:vt_loc}
v_t(\gamma_x(t)) = \int_0^t \mathfrak{F}^{t,r}_{\mathfrak{u},\gamma_x(r)} \left( f_r(\gamma_x(r)) \right)\,dr \equiv  \int_0^tP_x(t,r)g_x(r)\,dr=:V_x(t).\end{align} 
where the identification is via the trivialisation and  $P_x(t,r)\in\mathcal L(\mathbb{A},\mathbb{A})$ is the local representative of $\mathfrak{F}^{t,r}_{\mathfrak{u},\gamma_x(r)}\colon A_{\gamma_x(r)} \rightarrow A_{\gamma_x(t)}$, and $g_x(r)\in\mathbb{A}$ of $f_r(\gamma_x(r))$. 
By \eqref{eq:LXu_local}, $P_x(t,r)$ satisfies 
\begin{align}\label{another:ODE}
\frac{\partial}{\partial t}P_x(t,r) = B_x(t)\circ P_x(t,r), \qquad P_x(r,r)=\operatorname{id}_{\mathbb A}
\end{align} 
where $B_x(t).\alpha = d_{\gamma_x(t)}\mathsf{u}_t \left( r_{\gamma_x(t)}(\alpha) \right) + C_{\gamma_x(t)} \left( \alpha,\mathsf{u}_t(\gamma_x(t)) \right)$.
By the definition of $B_x(t)$, the $C^{k,\infty}$-regularity of $\mathfrak{u}$, and \Cref{rem:Bastiani_vs_Frechet} show that $t\mapsto B_x(t)\in L(\mathbb{A},\mathbb{A})$ is a continuous curve. We use the standard Leibniz rule for parameter-dependent Bochner integrals (which follows from \cite[Theorem 25.15 and Theorem 25.21]{Sche97}) together with \eqref{another:ODE}:
\begin{align*} 
\frac{d}{dt} V_x(t) &= g_x(t) + \int_0^t B_x(t).P_x(t,r)g_x(r)\,dr 
=g_x(t) + B_x(t).V_x(t).
\end{align*}
where the last equality exploits that $B_x(t)$ for fixed $t$ is bounded and linear, whence it commutes with the Bochner integral. Evaluating $B_x(t)$ on $V_x$, the last equation together with the definitions lead to:
$$\frac{d}{dt}V_x(t)=f_t(\gamma_x(t)) +  d_{\gamma_x(t)}\mathsf{u}_t \left( r_{\gamma_x(t)} (v_t(\gamma_x(t))) \right)+ C_{\gamma_x(t)} \left( v_t(\gamma_x(t)), \mathsf{u}_t(\gamma_x(t)) \right)$$
Since $v$ is a $C^{1,\infty}$-section, the chain rule applied to \eqref{eq:vt_loc} yields
$\frac{d}{dt}v_t(\gamma_x(t))= \left(\partial_tv_t+dv_t(r(\mathsf{u}_t)) \right)(\gamma_x(t)).$ 
Inserting both identities in \eqref{eq:vt_loc} and solving for $\frac{\partial v_t}{\partial t}$ we obtain along $\gamma_x(t)$
\begin{align*}
\frac{\partial v_t}{\partial t} = f_t + d\mathsf{u}_t(r(v_t)) - dv_t(r(\mathsf{u}_t)) + C(v_t,\mathsf{u}_t) \stackrel{\eqref{loctrivrbracket}}{=} f_t+ \LB[v_t,\mathsf{u}_t]_A. \end{align*}
As the initial value $v_0=0$ is immediate, we see that $v$ solves \eqref{eq:TD_pushforward}.

Uniqueness follows as in the proof of 1.: Let $z_t$ be another $C^{1,\infty}$-solution of \eqref{eq:TD_pushforward}. Fix $x\in U$ and define $ Z_x(t):=z_t(\gamma_x(t))$ identified as a curve in $\mathbb{A}$. Using the chain rule and \eqref{eq:TD_pushforward} for $z_t$, we see that  
$\frac{d}{dt} Z_x(t) = g_x(t)+B_x(t).Z_x(t), \ Z_x(0)=0$ 
Thus $Z_x$ and $V_x$ solve the same initial value problem in $\mathbb{A}$. Uniqueness for linear ODEs on Banach spaces gives $Z_x(t)=V_x(t).$ We obtain $z_t(\gamma_x(t)) = v_t(\gamma_x(t))$. As $\varphi_{t,s}$ is locally invertible, every point in the relevant time-$t$ flow domain is of the form $\gamma_x(t)=\varphi_{t,0}(x)$, whence $z_t=v_t$. This establishes the desired properties in the case the bundle $A_U$ is trivial.

For the general case, i.e. that there is not a global trivialisation, we subdivide the compact time intervals into finitely many subintervals each of which is contained in a trivialization. The local equations and uniqueness will then patch across the subdivision points. This is due to the fact that the differential equation are given by right-hand sides which are intrinsic (to see this note that the operators are local representatives of the fourth component of $\mathcal{X}_{\mathfrak{u}_t}$ (cf. \eqref{eq:LXu_local}). Hence all assertions are local on the base and generalise to the case where we need partitions of time intervals and several trivialisations.

\section{Dual linear fields and Poisson structures}\label{app:dualfield}
The material in this appendix is supplementary and is not used in the proofs of the main results. We mention that there is a dual theory of (local) linear vector fields on the dual bundle $A^\ast$. The exposition here follows \cite{Mac05} but needs the notion of a partial almost Poisson structure from \cite{CaPe12,CaPe20}. Also in the Banach setting a version of the classical result about equivalence between Poisson structure on $E^*$ and structure of Lie algebroid on $E$ (see for example \cite{Mar02} in finite dimension) holds, we refer the reader to \cite[Theorem 4.8]{CaPe12}.

\begin{defi}
For a weak  Banach  subbundle $p^{\f}_M\colon T^{\flat}M\to M$ of $p_M\colon T^*M\to M$, a bundle morphism $P\colon T^{\flat}M\to TM$ will be called a \emph{partial almost Poisson anchor} if $P$ is skew-symmetric relatively to the duality pairing.
\end{defi}

Let $P\colon T^{\flat}M\to TM$ be a partial almost Poisson anchor and $\iota\colon  T^{\flat}M\to T^*M$ the canonical inclusion. For any open set $U$ in $M$, let $\mathcal{F}^{\f}(U)$ be the set of smooth functions $f\colon U\rightarrow\R$ such that $df$ is a section of
$\pi^{\f}\colon T^{\f}M_{U}\rightarrow U$. It is clear that $\mathcal{F}^\flat
(U)$ is a sub-algebra of the algebra $\mathcal{F}(U)$ of
smooth functions on $U$. 
Then
\begin{equation}\label{dualbrac}
\{f,g\}_{P}=-\langle df,P(dg)\rangle 
\end{equation}
 defines a skew-symmetric bilinear map.

\begin{defi}
Let $P\colon T^\f M\to TM$ be a partial almost Poisson anchor.
\begin{enumerate}
\item The bilinear map $\{\cdot,\cdot\}_P$ is called the \emph{almost Poisson bracket} associated to $P$ and $(M,\{\;,\;\}_{P})$ is called an \emph{almost  partial Poisson structure} on $M$
 \item A partial almost Poisson structure $(M,\{\;,\;\}_{P})$ is called a \emph{partial Poisson manifold} if the bracket $\{\;,\;\}_{P}$ satisfies the Jacobi identity 
In this case $P$ is called a \emph{Poisson anchor}.
\end{enumerate}
\end{defi}

Using the dualities and Poisson structures one can define (local) linear vector fields on $A_{U}^*$. For any open set $U$ in $M$ and  any $f\in \mathcal{F}^\f(A^*_{U})$ the vector field  $P({df})$ on $A^*_{U}$ is well defined. Therefore for any $\mathfrak{u}\in \G(A_{U})$ we can associate the vector field $\Xi_\mathfrak{u}=P(d\Phi_\mathfrak{u})$ on $A^*_{U}$ such that $T\pi_{*}(\Xi_\mathfrak{u})=\rho(\mathfrak{u})$. As  $\mathcal{F}_L(A^*_{U})$ is contained in $\mathcal{F}^\f(A^*_{U})$ ,  for any $f,f'\in\mathcal{F}_L(A^*_{U})$ the almost Lie bracket $\{f,f'\}_P$ is also a linear function.  Since $P(df)(f')=\{f,f'\}_P$, it follows that $\Xi_\mathfrak{u}(\mathcal{F}_L(A^*_{U})\subset \mathcal{F}_L(A^*_{U})$. Now from \cite[Theorem 4.8]{CaPe12} we also have $\Xi_\mathfrak{u}(\mathcal{F}_{\pi_{*}}(A_{U}))\subset \mathcal{F}_{\pi_{*}}(A_{U})$. 

\begin{lem}\label{lineXi}
Let  $(A,M, \rho,\LB _A)$ be  a Banach  almost Lie algebroid.
For any local section $\mathfrak{u}\in \G(A_{U})$, the pair $(\Xi_\mathfrak{u}, \rho(\mathfrak{u}))$ is a linear vector field on $A^*$. Moreover, for any $\mathfrak{v}\in  \G(A_{U})$, we have 
\begin{eqnarray}\label{Xivalue}
\Xi_\mathfrak{u}(\Phi_{\mathfrak{v}})=\Phi_{[\mathfrak{u},\mathfrak{v}]}.
\end{eqnarray}
\end{lem}
\begin{proof} That $(\Xi_\mathfrak{u}, \rho(\mathfrak{u}))$ is a linear vector field on $A^*$ follows from \Cref{chalinX}.
It remains to prove \eqref{Xivalue}. For almost partial Hamiltonian fields, relations \eqref{dualbrac} and \cite[Theorem 4.8]{CaPe12} yields:
$\Xi_\mathfrak{u}(\Phi_{\mathfrak{v}})=\{\Phi_\mathfrak{u},\Phi_\mathfrak{v} \}_{P}= \Phi_{[\mathfrak{u},\mathfrak{v}]}.$
\end{proof}
\begin{rem}[Dualities and dual pairings]\label{remark:duality}
Assume that $(A,M,\rho,\LB _A)$ is a Banach almost Lie algebroid with $P\colon T^\f A^*\to TA^*$ a linear partial almost Poisson anchor and $\{\cdot,\cdot\}_P$ almost Poisson bracket on $\mathcal{F}^\f(E^*_{U})$ for any open set $U$.  The smooth dual (with respect to $P$) admits a convenient generating set:
\begin{enumerate}
\item[1.]  For $\mathfrak{v}\in A$,  $T^\f_{\mathfrak{v} }A^\ast$ is generated by 
\[\{d\Phi^\s(\mathfrak{v}), \s \in \G(A^*_{U})\} \cup \{d(f\circ \pi)(\mathfrak{v}), f \in \mathcal{F}(U)\}.\]
\end{enumerate}
As $T\pi\colon TA\to TM$ and  $T\pi_*\colon TA^*\to TM$ are vector bundles we construct a dual pairing between $TA$ and $T^\flat A^*$ as bundles over $TM$ (similar to the finite dimensional setting \cite{Mac05}): Denote by $T_{(x,y)}A$ and $T_{(x,y)}A^*$ the fiber $(T\pi)^{-1}(x,y)$ and $(T\pi_*)^{-1}(x,y)$ for  $(x,y)\in TM$. 
If $\mathcal{X}=(x,y, u,v)\in T_{(x,y)}A$ and $\Xi=(x,y,\xi,\eta)\in T_{(x,y)}A^*$ let $c\colon J\rightarrow A$ and $c_*\colon J \rightarrow A^\ast$ be  $C^1$ curves defined on some interval $J$ with $\pi\circ c=\pi_*\circ  c_*$ and $\dot{c}(0)=\mathcal{X}, \dot{c}_*(0)=\Xi$, we set:
\begin{align}\label{Dualpairing}
\DP[\Xi,\mathcal{X}]=\left.\dis\frac{d}{dt}\right|_{t=0}\langle c_*(t),c(t)\rangle
\end{align}
Using arguments as in \cite[Proofs of Proposition 3.4.6 and 3.4.7]{Mac05} the following is true for Banach manifolds:
\begin{enumerate}
\item[2.] The pairing \eqref{Dualpairing} is non degenerate and identifies $TA_U$ with a canonically paired (weak Banach) subbundle $T^\flat A_U^\ast\subseteq TA_U^\ast$. For linear vector fields $(\mathcal{X}, X)$ and $(\Xi, X)$ on $A_{ U}$ and $A^*_{U}$ and  $p_A(\mathcal{X})=\mathfrak{v}$ and $p_{A^*}(\Xi)=\s$,
$$\DP[\Xi,\mathcal{X}]=\Xi(\Phi_\mathfrak{v})+\mathcal{X}(\Phi^\s)- d(\langle\s, \mathfrak{v}\rangle)(X).$$
By \eqref{Dualpairing} linear vector fields on $A$ and $A^\ast$ are in correspondence.
\end{enumerate}
\end{rem}

We prove the following dual statement to \Cref{thm:complete-lift}:

\begin{prop} Let $(A,\pi,M,\rho,\LB_A)$ be a Banach almost Lie algebroid and $\mathfrak{u}\in\Gamma(A_U)$ a local section.  Let $(\Xi_\mathfrak{u}:=P(d\Phi_\mathfrak{u}),\rho(\mathfrak{u}))$ be the linear vector field on $A^*$ associated to $\mathfrak{u}$. Its local flow is $\mathfrak{F}^t_{\Xi_\mathfrak{u},x} = \left((\mathfrak{F}^t_{\mathfrak{u},x})^T\right)^{-1}, x\in U$.
\end{prop}

\begin{proof}
By \Cref{lineXi}, the pair $(\Xi_\mathfrak{u},\rho(\mathfrak{u})),\text{ with } \Xi_u:=P(d\Phi^\mathfrak{u})$, is a linear vector field on $A^*_U$. Its local flow will be denoted by $\mathfrak{F}^t_{\Xi_\mathfrak{u}}$. Restricting the flow to a fibre $A_x^*$ we obtain a map $ \mathfrak{F}^{t,x}_{\Xi_\mathfrak{u}}\colon A_x^*\to A_{\varphi_t(x)}^*$, where $\varphi_t$ is the flow of $\rho(\mathfrak{u})$, following the notation from the proof of 2. in \Cref{thm:complete-lift}.

Fix $x\in U$, $a\in A_x$, and $\xi\in A_x^*$. Choose local sections $\mathfrak{v}\in\Gamma(A_V),  \sigma\in\Gamma(A^*_V)$ defined on an $x$-neighbourhood $V$, such that $\mathfrak{v}(x)=a,  \sigma(x)=\xi$. By the formula \eqref{Dualpairing} for the tangent pairing of linear vector fields from \Cref{remark:duality}, we have 
\begin{align}\label{eq:dualpairingforXLU}
\DP[\Xi_\mathfrak{u}(\xi),\LXu(a)] = \Xi_\mathfrak{u}(\Phi_\mathfrak{v})(\xi) + \LXu(\Phi^\sigma)(a)- d_x\langle\sigma,\mathfrak{v}\rangle\bigl(\rho(\mathfrak{u})(x)\bigr).
\end{align}
\Cref{lineXi} gives $\Xi_\mathfrak{u}(\Phi_\mathfrak{v})=\Phi_{\LB[\mathfrak{u},\mathfrak{v}]_A}$, and by \eqref{Xs} $\LXu(\Phi^\sigma)=\Phi^{L_\mathfrak{u}\sigma}$.
Evaluating at $x$, we insert these expressions in \eqref{eq:dualpairingforXLU} and use \eqref{Lsigma}. Then 
\begin{align*}  &\DP[\Xi_\mathfrak{u}(\xi),\LXu(a)] \\ =& \langle\xi,\LB[\mathfrak{u},\mathfrak{v}]_A(x)\rangle+d_x\langle\sigma,\mathfrak{v}\rangle(\rho(\mathfrak{u})(x)) - \langle\xi,\LB[\mathfrak{u},\mathfrak{v}]_A(x)\rangle- d_x\langle\sigma,\mathfrak{v}\rangle(\rho(\mathfrak{u})(x)) =0.
\end{align*}
Thus $\LXu$ and $\Xi_\mathfrak{u}$ are dual linear vector fields in the sense of \Cref{remark:duality}. 
Now for any pair of elements $(\mathfrak{v},\s)$ such that $x=\pi(\mathfrak{v})=\pi_*(\s)$ the flows  $\mathfrak{F}_{\Xi_\mathfrak{u}}^t(\s)$  and  $\mathfrak{F}_{\mathfrak{u}}^t(\mathfrak{v})$ are defined if and only if the flow $\varphi_t(x)$ is defined and so we have 
\begin{align*}
\langle\mathfrak{F}_{\Xi_\mathfrak{u}}^{t}(\s), \mathfrak{F}_{\mathfrak{u}}^{t}(\mathfrak{v})\rangle =\langle (\mathfrak{F}_{\mathfrak{u}}^{t})^T\circ  \mathfrak{F}_{\Xi_\mathfrak{u}}^{t}(\s), \mathfrak{v}\rangle \end{align*}
By the definition, $\frac{d}{dt}\langle \mathfrak{F}_{\Xi_\mathfrak{u},x}^t(\xi),\mathfrak{F}^t_{\mathfrak{u},x}(a)\rangle = \DP[\Xi_\mathfrak{u}(\mathfrak{F}_{\Xi_\mathfrak{u},x}^t(\xi)),\LXu(,\mathfrak{F}^t_{\mathfrak{u},x}(a))] = 0$. Hence $\left\langle (\mathfrak{F}_{\mathfrak{u},x}^t)^T \mathfrak{F}^t_{\Xi_\mathfrak{u},x}(\xi), a \right\rangle = \langle\xi,a\rangle$ for every $a\in A_x$.  It follows that $(\mathfrak{F}_{\mathfrak{u},x}^{t})^T\circ  \mathfrak{F}_{\Xi_\mathfrak{u},x}^{t}=\mathrm{id}_{A^\ast_x}$ and so
$\mathfrak{F}_{\Xi_\mathfrak{u}}^{t}=((\mathfrak{F}_{\mathfrak{u}}^{t})^T)^{-1}.$
\end{proof}

\addcontentsline{toc}{section}{References}
\bibliography{Banach_algebroid}
\end{document}